\documentclass[11pt, reqno]{article}

\usepackage{fullpage}
\usepackage{enumerate}
\usepackage{amsmath,amsfonts,amssymb,graphics,amsthm}
\usepackage{hyperref}
\hypersetup{
    colorlinks=true,
    linkcolor=blue,
    citecolor=red,
    urlcolor=blue,
    pdfborder={0 0 0}
}

\usepackage{comment}

\usepackage{graphicx}
\usepackage{xcolor}

\usepackage[font=sf, labelfont={sf,bf}, margin=1cm]{caption}

\usepackage{bbm}
\usepackage{cleveref}
  \crefname{theorem}{Theorem}{Theorems}
  \crefname{thm}{Theorem}{Theorems}
  \crefname{lemma}{Lemma}{Lemmas}
  \crefname{lem}{Lemma}{Lemmas}
  \crefname{remark}{Remark}{Remarks}
  \crefname{prop}{Proposition}{Propositions}
\crefname{conjecture}{Conjecture}{Conjectures}
  \crefname{defn}{Definition}{Definitions}
\crefname{claim}{Claim}{Claims}
  \crefname{corollary}{Corollary}{Corollaries}
  \crefname{section}{Section}{Sections}
  \crefname{figure}{Figure}{Figures}

\newtheorem{thm}{Theorem}[section]
\newtheorem{lemma}[thm]{Lemma}
\newtheorem{corollary}[thm]{Corollary}
\newtheorem{prop}[thm]{Proposition}
\newtheorem{defn}[thm]{Definition}

\newtheorem{conjecture}[thm]{Conjecture}

\numberwithin{equation}{section}

\theoremstyle{definition}
\newtheorem{remark}[thm]{Remark}
\def \aH {\textsf{H}}
\def \ah {\textsf{h}}
\def \aC {\textsf{C}}
\def \ac {\textsf{c}}
\def \aF {\textsf{F}}
\def \ae {\mathsf{e}}
\def \len {\textsf{Len}}
\def \Area {\textsf{Area}}
\def \ph {\varphi}
\def \u {\text{upper}}

\def\cW{\mathcal{W}}

\def\cS{\mathcal{S}}
\def\cR{\mathcal{R}}

\def\cM{\mathcal{M}}
\def\cL{\mathcal{L}}

\def\cH{\mathcal{H}}
\def\cG{\mathcal{G}}
\def\cF{\mathcal{F}}
\def\cE{\mathcal{E}}

\def\cC{\mathcal{C}}
\def\cB{\mathcal{B}}

\def \ve {\varepsilon}

\def\P{\mathbb{P}}
\def\E{\mathbb{E}}

\def\R{\mathbb{R}}
\def\Z{\mathbb{Z}}
\def\N{\mathbb{N}}

\def \bv {\boldsymbol v}
\def \bd {\boldsymbol V}
\def  \p- {p\textunderscore}

\def\eps{\varepsilon}

\DeclareMathOperator{\var}{Var}

\DeclareMathOperator{\cov}{Cov}

\usepackage{tikz}

\newcommand{\note}[1]{{\textcolor{red}{[note: #1]}}}
\newcommand{\indic}[1]{\mathbf{1}_{\{#1\}}}

\def\ph{\varphi}

\begin{document}

\title{Critical exponents on Fortuin--Kasteleyn weighted planar maps}

\author{Nathana\"el Berestycki\thanks{Supported in part by EPSRC grants EP/L018896/1 and EP/I03372X/1} \and Beno\^it Laslier\thanks{Supported in part by EPSRC grant EP/I03372X/1} \and Gourab Ray\thanks{Supported in part by EPSRC grant EP/I03372X/1}\\
}
\date{\small \today}

\maketitle

\medskip

\begin{abstract}
In this paper we consider random planar maps weighted by the self-dual Fortuin--Kasteleyn model with parameter $q \in (0,4)$. Using a bijection due to Sheffield and a connection to planar Brownian motion in a cone we obtain rigorously the value of the critical exponent associated with the length of cluster interfaces, which is shown to be
$$
 \frac{4}{\pi} \arccos \left( \frac{\sqrt{2 - \sqrt{q}}}{2} \right)=\frac{\kappa'}{8}.
$$
where $\kappa' $ is the SLE parameter associated with this model.
We also derive the exponent corresponding to the area enclosed by a loop which is shown to be 1 for all values of $q \in (0,4)$. Applying the KPZ formula we find that this value is consistent with the dimension of SLE curves and SLE duality. 
\end{abstract}



\def\m{\mathbf{m}}
\def\t{\mathbf{t}}
\def\p{\partial}
\def \bm {\boldsymbol{m}}
\def \bt {\boldsymbol{t}}

\tableofcontents
\begin{figure}[h]
  \centering
 \includegraphics[width=.49\textwidth]{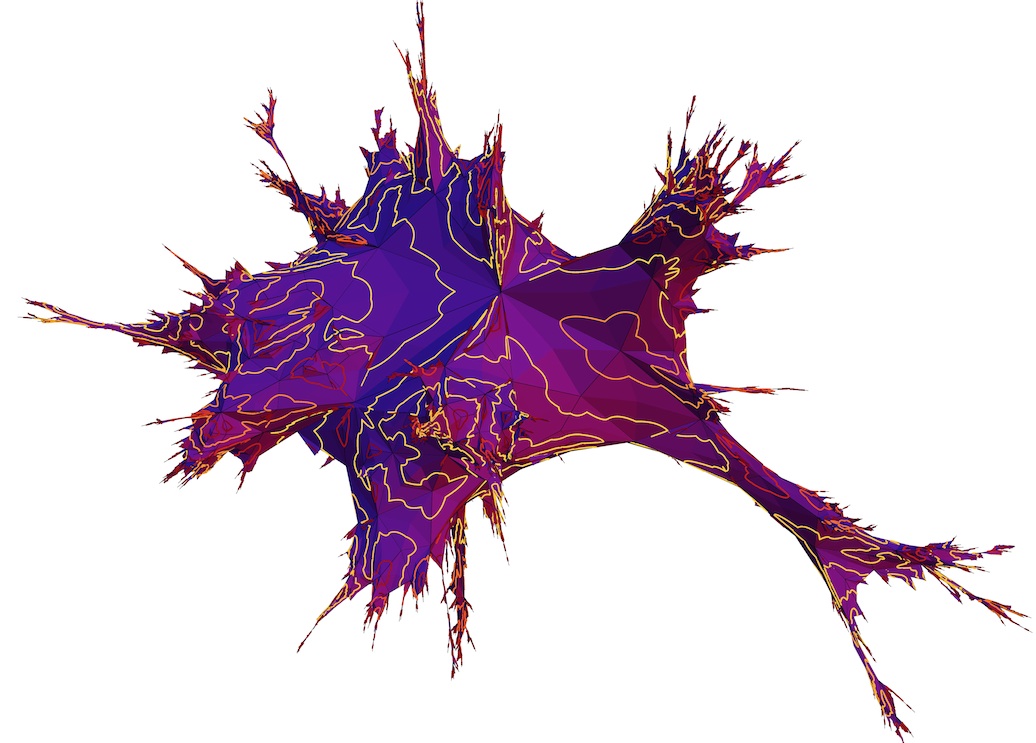}
\includegraphics[width = .49 \textwidth]{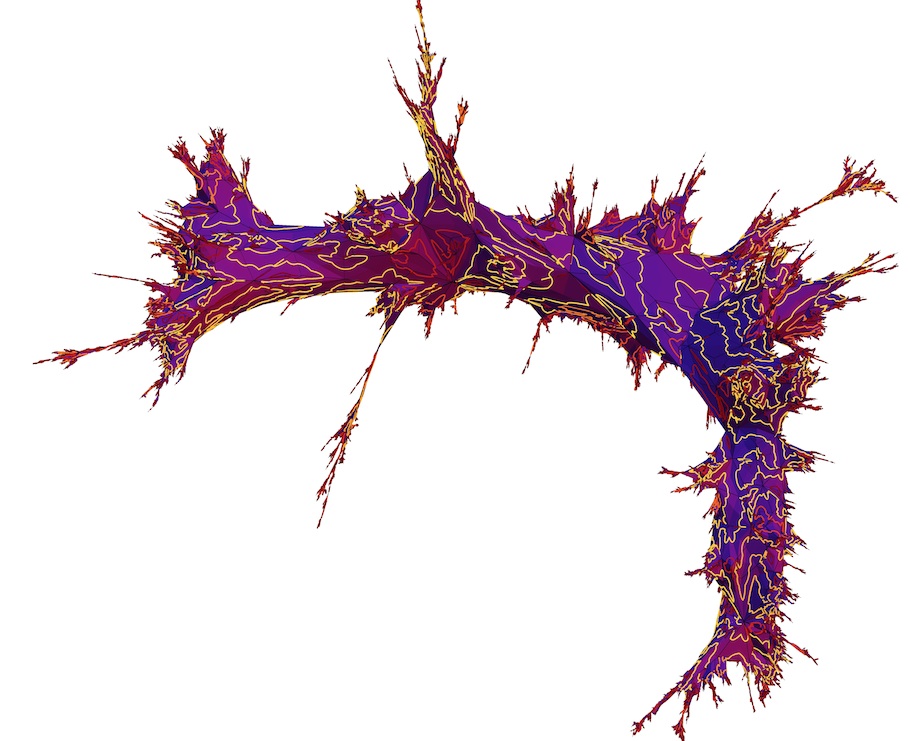}
\caption{FK-weighted random map and loops for $q=0.5$ (left) and $q=2$ (corresponding to the Ising model, right). The shade of loops indicates their length (dark for short and light for long loops).}
\end{figure}
\section{Introduction}

\label{S:intro}
Random surfaces have recently emerged as a subject of central
importance in probability theory. On the one hand, they are connected
to theoretical physics (in particular string theory) as they are basic
building blocks for certain natural quantizations of gravity
\cite{Pol, David,kpz88,DistKa}. On the other hand, at the mathematical
level, they show a very rich and complex structure which is only
beginning to be unravelled, thanks in particular to recent parallel
developments in the study of conformally invariant random processes,
Gaussian multiplicative chaos, and bijective techniques. We refer to
\cite{bourbaki} for a beautiful exposition of the general area with a
focus on relatively recent mathematical developments.

This paper is concerned with the geometry of random planar maps, which
can be thought of as canonical discretisations of the surface of
interest. The particular distribution on planar maps which we consider
was introduced in \cite{She11} and is roughly the following (the detailed
definitions follow in \cref{sec:critical-fk-model}). Let $q<4$ and let $n \ge 
1$. The random map $M_n$ that we consider is decorated with a (random) subset 
$T_n$  of edges. The map $T_n$ induces a dual collection of edges $T_n^\dagger$ 
on the dual map of $M$ (see \cref{fig:clusters}). Let $\m$ be a planar map with 
$n$ edges, and $\t$ a given subset of edges of $\m$. Then the probability to 
pick a particular $(\m,\t)$ is, by definition, proportional to 
\begin{equation}\label{FK}
\P( M_n = \m, T_n= \t) \propto \sqrt{q}^\ell,
\end{equation}
\begin{figure}[h]
\centering{
1.\includegraphics[scale=.4]{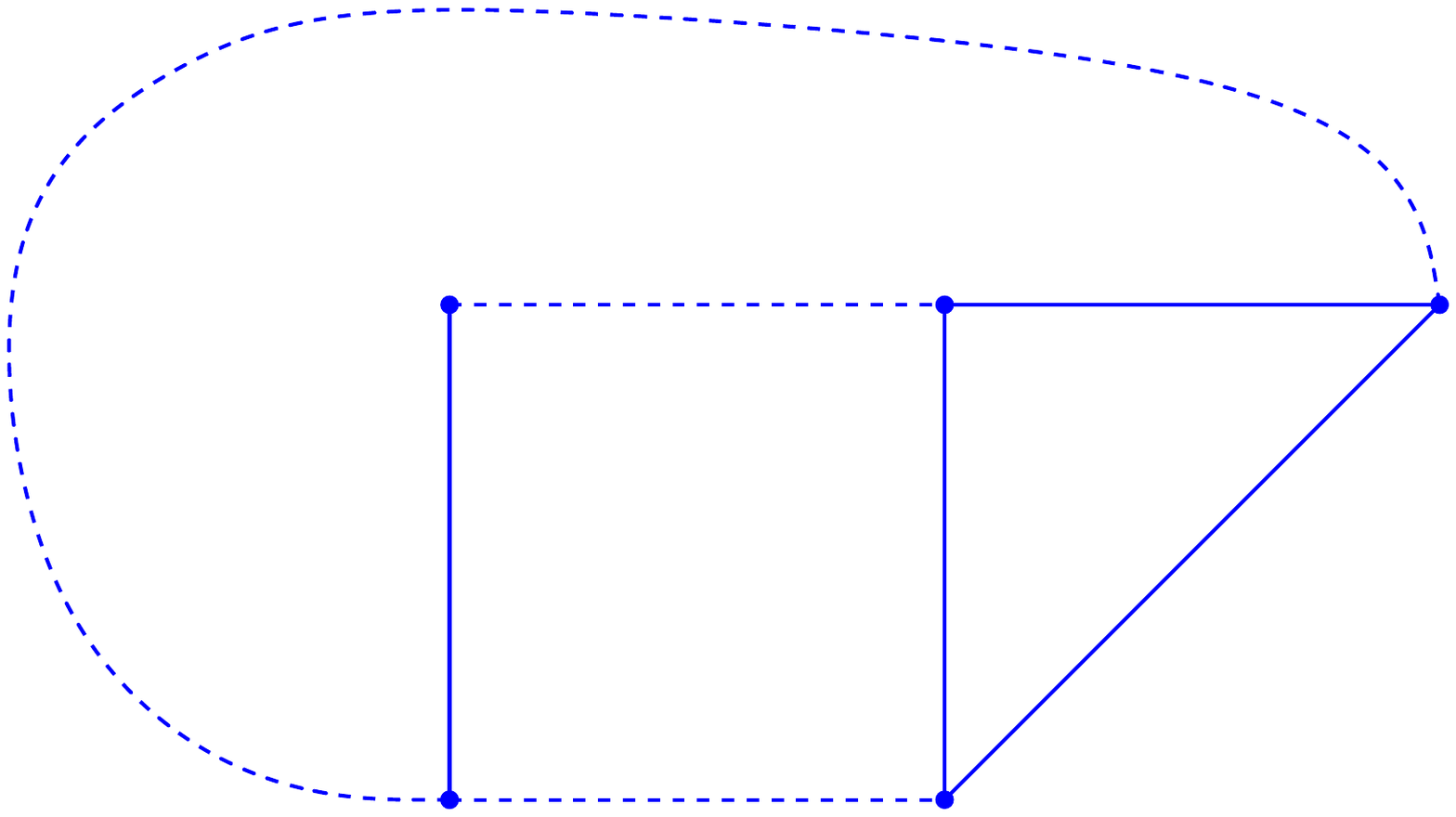} \quad \quad
2. \includegraphics[scale=.4]{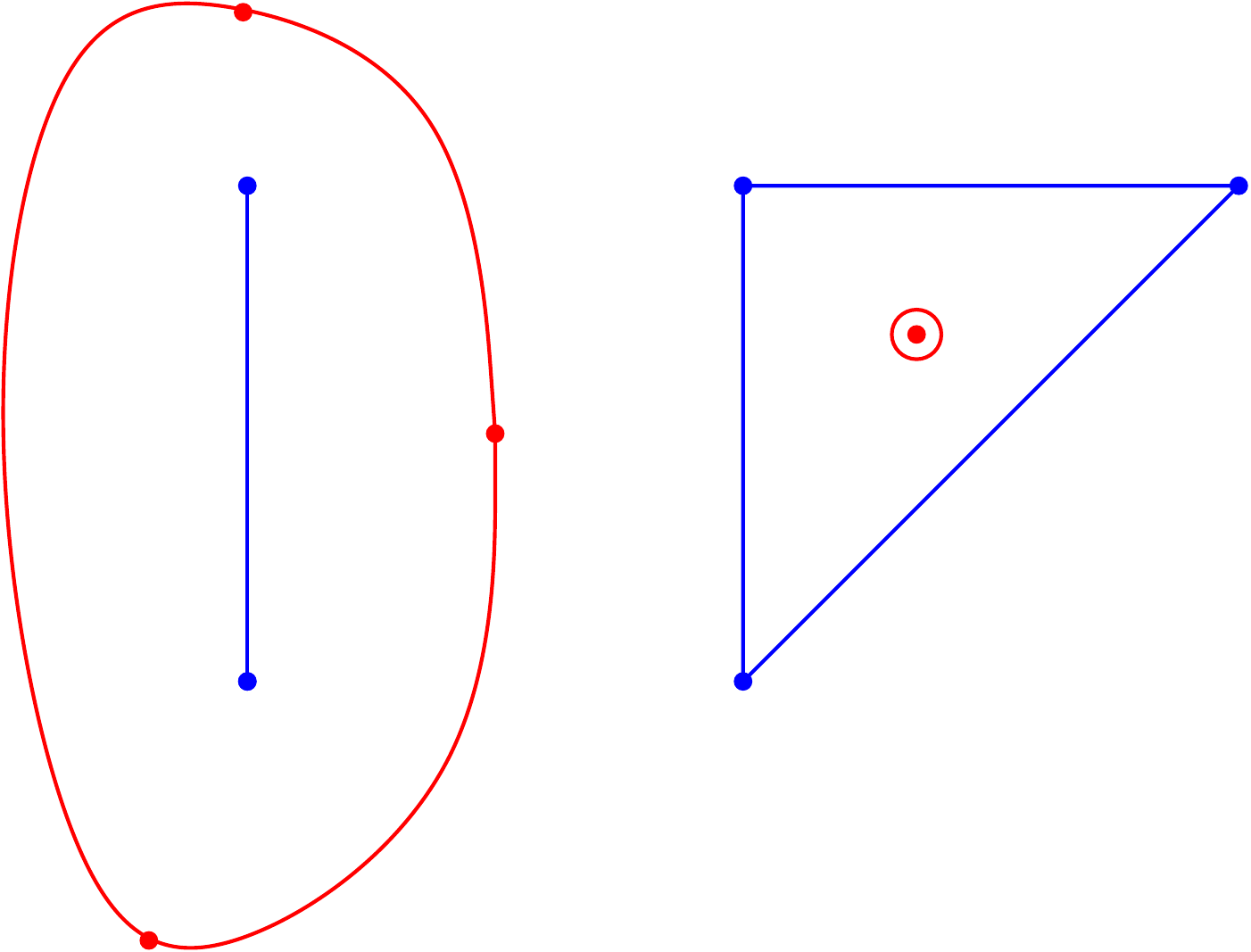}
}

\vspace{.5cm}
\centering{
3. \includegraphics[scale=.4]{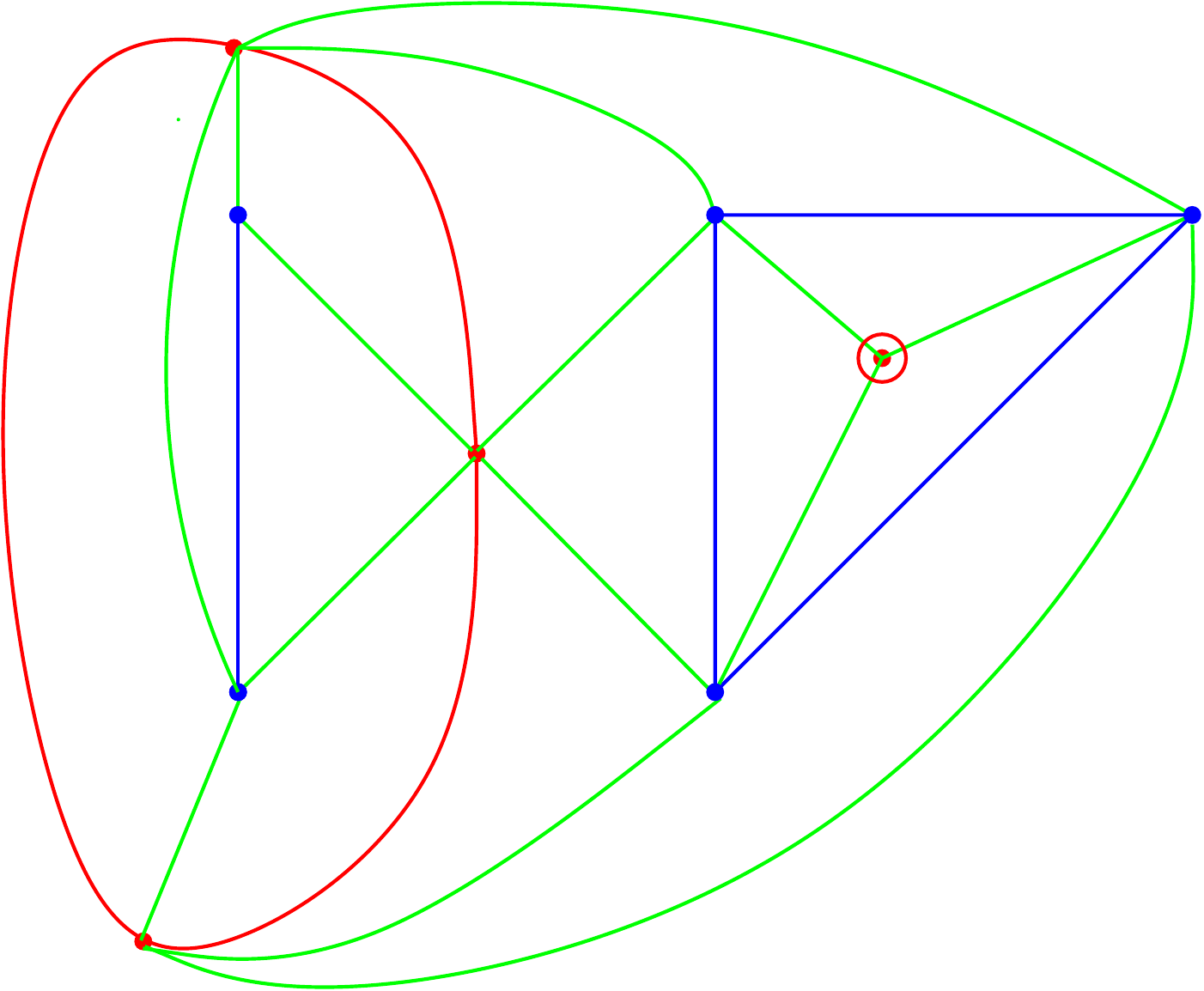}\quad \quad
4. \includegraphics[scale=.4]{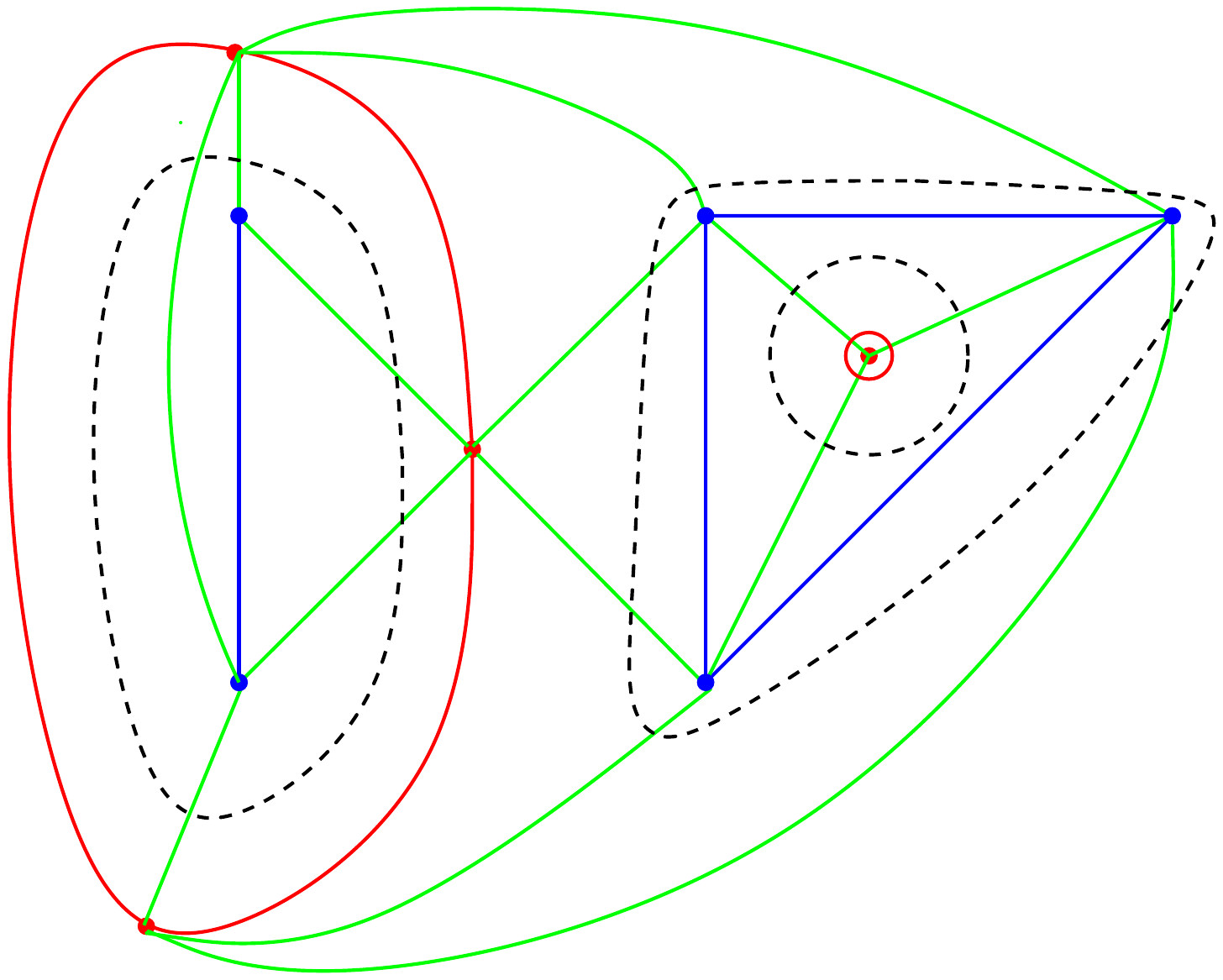}
}
\caption{A map $m$ decorated with loops associated with a set of open edges $t$. 
\emph{Top left:} the map is
  in blue, with solid open edges and dashed closed edges. \emph{Top right:} Open clusters and corresponding open dual clusters in blue and red. \emph{Bottom left:} every dual vertex is joined to its adjacent primal vertices by a green edge. This results in an augmented map $\bar m$ which is a triangulation. \emph{Bottom right:} the primal and dual open clusters are separated by loops, which are drawn in black and are dashed. Each loop crosses every triangle once, and so can be identified with the set of triangles it crosses. See \cref{sec:critical-fk-model} for details.
%
 %
 }
\label{fig:clusters}
\end{figure}
where $\ell$ is the (total) number of loops in between both primal and
dual vertex clusters in $\t$ which is equal to
the combined number of cluster in $T_n$ and $T_n^\dagger$ minus $1$
(details in \cref{sec:critical-fk-model}). Equivalently
given the map $M_n = \m$, the collection of edges $T_n$ follows the
distribution of the self-dual Fortuin--Kasteleyn model, which is in
turn closely related to the critical $q$-state Potts model, see
\cite{BKW}. Accordingly, the map $M_n$ is chosen with probability
proportional to the partition function of the Fortuin-Kasteleyn
model on it.

One reason for this particular choice is the belief (see
e.g. \cite{LQGmating}) that after Riemann uniformisation, a large sample of
such a map closely approximates a \emph{Liouville quantum gravity} surface. 
This is the random metric obtained by considering the Riemannian metric tensor
\begin{equation}
\label{tensor}
e^{\gamma h(z)} |dz|^2,
\end{equation}
where $h(z)$ is an instance of the Gaussian free field. (We emphasise that a rigorous construction of the metric associated to \eqref{tensor} is still a major open problem.) The parameter $\gamma\in(0,2)$ is then believed to be related to the parameter $q$ of \eqref{FK} by the relation
\begin{equation}
\label{q_gamma}
q = 2 + 2\cos\left(\frac{8\pi}{\kappa'}\right); \quad \gamma = \sqrt{\frac{16}{\kappa'}} .
\end{equation}
Note that when $q \in (0,4)$ we have that $\kappa' \in (4,8)$ so that
it is necessary to generate the Liouville quantum gravity with the
associated dual parameter $\kappa  = 16/\kappa' \in (0,4)$. This
ensures that $\gamma = \sqrt{\kappa } \in (0,2)$, which is the nondegenerate phase for the associated mass measure and Brownian motions,
see \cite{DS11kpz,ber2013diffusion,berkpz}.

Observe that when $q=1$, the FK model reduces to ordinary bond
percolation. Hence this corresponds to the case where $M$ is chosen
according to the uniform probability distribution on planar maps with
$n$ edges. This is a situation in which remarkably detailed
information is known about the structure of the planar map. In
particular, a landmark result due to Miermont \cite{Miermont} and Le
Gall \cite{LeGall} is that, viewed as a metric space, and rescaling
edge lengths to be $n^{-1/4}$, the random map converges to a  multiple
of a certain universal random metric space, known as the
\emph{Brownian map}. (In fact, the results of Miermont and Le Gall
apply respectively to uniform quadrangulations with $n$ faces and to
$p$-angulation for $p=3$ or $p$ even, whereas the convergence result
concerning uniform planar maps with $n$ edges was established a bit
later by Bettinelli, Jacob and Miermont
\cite{betscaling}).  Critical percolation on a related
half-plane version of the maps has been analysed in a recent work of
Angel and Curien \cite{AC13}, while information on the full
plane percolation model was more recently obtained by Curien and
Kortchemski \cite{CK13}. Related works on loop models (sometimes
rigorous, sometimes not) appear in \cite{guionnet12,borot12,eynard95,b11,bor12loop, BBG}.

The goal of this paper is to obtain detailed geometric information about the clusters of the self-dual FK model in the general case $q\in (0,4)$. As we will see, our results are in agreement with nonrigorous predictions from the statistical physics community. In particular, after applying the \textbf{KPZ transformation}, they correspond to Beffara's result about the dimension of SLE curves \cite{beffdim} and SLE duality. 

\subsection{Main results}
Let $L_n$ denote a \emph{typical loop}, that is, a loop chosen
uniformly at random from the set of loops induced by $(M_n,T_n)$ which
follow the law given by \eqref{FK}. Such a loop separates the
map into an outside component which contains the root and an inside
component which does not contain the root (precise definitions follow in
\cref{sec:critical-fk-model}). If the loop passes through the root, we
leave $L_n$ undefined (this is a low probability event so the
definition does not
matter). Let $\len(L_n)$ denote the number of
triangles in the loop and let $\Area(L_n)$ denote the number of
triangles inside it. Let
\begin{equation}\label{p0}
p_0 = \frac{\pi}{4 \arccos \left( \frac{\sqrt{2 - \sqrt{q}}}{2}
  \right) } =\frac {\kappa'}{8}
  \end{equation}
where $q$ and $\kappa'$ are related as in \eqref{q_gamma}.

\begin{thm}
\label{T:typ}
We have that $\len(L_n) \to \mathsf L$ and $\Area(L_n) \to \mathsf{A}$
in law. Further, the random variables $\mathsf L$ and $\mathsf A$
satisfy the following. 
\begin{equation}\label{Tboundary}
  \P( \mathsf L> k) = k^{-1/p_0 + o(1)},
\end{equation}
and
\begin{equation}\label{Tarea}
 \P( \mathsf A > k) =  k^{-1 + o(1)}.
\end{equation}
\end{thm}
\begin{remark}
As we were finishing this paper, we learnt of the related work, completed 
independently and simultaneously, by Gwynne, Mao and Sun \cite{GMS}. They obtain 
several scaling limit results, showing that various quantities associated with 
the FK clusters converge in the scaling limit to the analogous quantities 
derived from Liouville quantum gravity in \cite{LQGmating}.  
Some of their results also overlap with the results above. In
particular they obtain a stronger version of the length exponent \eqref{Tboundary} 
by showing
that in addition that the tails are regularly varying. Though both papers rely
on Sheffield's bijection \cite{She11} and a connection to planar Brownian 
motion in a cone, it is interesting to note that the proof techniques are 
substantially different. The techniques in this paper are comparatively simple, 
relying principally on harmonic functions and appropriate martingale techniques.
\end{remark}

Returning to \cref{T:typ}, it is in fact not so hard to see
that when rooted at a randomly chosen edge, the decorated maps
$(M_n,T_n)$ themselves converge for the Benjamini--Schramm (local)
topology. This is already implicit in the work of Sheffield \cite{She11}
and properties of the infinite local limit $(M_\infty, T_\infty)$ have
recently been analysed in a paper of Chen \cite{chen}. In particular a uniform exponential bound on the
degree of the root is obtained. Together with earlier results of Gurel Gurevich
and Nachmias \cite{GN12}, this implies for instance that random walk on
$M_\infty$ is a.s. recurrent. From this it is actually not hard to see
that $\len (L_n)$ and $\Area(L_n)$ converge in law in
\cref{T:typ}. The major contributions in this paper are the other
assertions in \cref{T:typ}.

\medskip Our results can also be phrased for the loop $L^*$ going through the
origin in this infinite map $M_\infty$. Since the root is uniformly
chosen from all possible oriented edges, it is easy to see that this
involves biasing by the length of a typical loop. Hence the exponents are slightly different. For instance,
for the length $\len(L^*)$ and $\Area(L^*)$ of $L^*$, we get
\begin{equation}
\label{lengthbias}
\P( \len(L^*) > k) = k^{-1/p_0+1 + o(1)},
\end{equation}
 For the area, it can be seen from our techniques that
\begin{equation}
\label{areabias}
\P( \Area(L^*) \ge k) = k^{- (1-p_0) + o(1)}, 
\end{equation}
(The authors of \cite{GMS} have kindly indicated to us that \eqref{areabias}, together with a regular variation statement, could probably also be deduced from their Corollary 5.3 with a few pages of work, using arguments similar to those already found in their paper).

\medskip  While our techniques could also probably be used to compute other related exponents we have not pursued this, in order to keep the paper as simple as possible. We also remark that the techniques in the present paper can be used to study the looptree structure of typical cluster boundaries (in the sense of Curien and Kortchemski \cite{CK13}). 

\begin{remark}
In the particular case of percolation on the uniform infinite random
planar map (UIPM) $M_\infty$, i.e. for $q=1$, we note that our results
give $p_0 = 3/4$, so that the typical boundary loop exponent is
$1/p_0= 4/3$. This is consistent with the more precise asymptotics
derived by Curien and Kortchemski \cite{CK13} for a
related percolation interface. Essentially their problem is analogous
to the biased loop case, for which the exponent is, as discussed
above, $1/p_0 -1 = 1/3$. This matches Theorem 1 in
\cite{CK13}, see also Theorem 2 (ii) in
\cite{AC13} for the half-plane case. Likewise, the exponent for
the area of $L^*$ (in the biased case) is $1-p_0 = 1/4$, which
matches (i) in the same theorem of \cite{AC13}.
\end{remark}
\subsection{Cluster boundary, KPZ formula, bubbles and dimension of SLE}
\label{SS:KPZ}
\textbf{KPZ formula}. We now discuss how our results verify the KPZ
relation between critical exponents. We first recall the 
KPZ formula. For a fixed or random independent set $A$ with Euclidean scaling exponent $x$, its ``quantum analogue" has a scaling exponent $\Delta$, where $x$ and $\Delta$ are related by the formula
\begin{equation}
\label{KPZ}
x = \frac{\gamma^2}{4} \Delta^2 + (1- \frac{\gamma^2}{4}) \Delta.
\end{equation}
See \cite{DS11kpz,berkpz,RhodesVargas} for rigorous formulations of this formula at the continuous level. Concretely, this formula should be understood as follows. Suppose that a certain subset $A$ within a random map of size $N$ has a size $|A| \approx N^{1-\Delta}$. Then its Euclidean analogue within a box of area $N$ (and thus of side length $n = \sqrt{N}$) occupies a size $|A'| \approx N^{1-x} = n^{1/2 - x/2}.$ In particular, observe that the approximate (Euclidean) Hausdorff dimension of $A'$ is then $ 2-2x$.
\paragraph{Cluster boundary.} The exponents in \eqref{Tboundary} and \eqref{Tarea} suggest that for a large critical FK cluster on a random map, we have the following approximate relation between the area and the length:
\begin{equation}
\label{iso1} 
\mathsf L  = \mathsf A^{p_0 + o(1)}.
\end{equation}
The relation \eqref{iso1}
suggests that the quantum scaling exponent $\Delta = 1-p_0$. Applying the KPZ formula we see that the corresponding Euclidean exponent is $1/2 -
\kappa'/16$. Thus the Euclidean dimension of the boundary is
$1+\kappa'/8$. The conjectured scaling limits of the boundary is a
CLE$_{\kappa'}$ curve and hence this exponent matches the one obtained by
Beffara \cite{beffdim}.

\paragraph{Bubble boundary.} We now address a different sort of
relation with its volume inside, which concerns large filled-in \emph{bubbles} or envelopes in the terminology which we use in this paper (see \cref{def:area} and immediately above for a definition). In the scaling limit and after a conformal embedding, these are expected to converge to filled-in SLE loops and more precisely, quantum discs in the sense of \cite{LQGmating}.
At the local limit level, they should correspond to Boltzmann maps whose boundaries should form a looptree structure 
in the sense of Curien and Kortchemski
\cite{CK13}. 
We establish in 
  \cref{thm:exit}, \cref{item4,mainitem} that with high probability
\begin{equation}
\label{iso2}
|\p \cB|= |\cB|^{1/2 + o(1)}.
\end{equation}
This suggests a quantum dimension of $\Delta = 1/2$ and
remarkably, this boundary bulk behaviour is independent of $q$ (or
equivalently of $\gamma$) and therefore corresponds with the usual
Euclidean isoperimetry in two dimensions. Applying the KPZ formula 
\eqref{KPZ}, we obtain a Euclidean scaling exponent $$x=\frac12 - \frac1{\kappa'}.$$
On the other hand, recall
the Duplantier duality which states that the outer boundary of an
SLE$_{\kappa'}$ curve is an SLE$_{16/{\kappa'}}$ = SLE$_\kappa$ curve. This has been
established in many senses in \cite{IGIV,dualitysle,zhanduality}. Hence the
dimension of the outer boundary should be $1+ \kappa/8 = 1+2/\kappa'$ which is equal to 
$2(1-x)$. Thus KPZ is verified.

\medskip \textbf{Acknowledgements} We are grateful to a number of people for useful discussions: Omer Angel, Linxiao Chen, Nicolas Curien, Gr\'egory Miermont, Jason Miller, Scott Sheffield, and Perla Sousi. We thank Ewain Gwynne, Cheng Mao and Xin Sun for useful comments on a preliminary draft, and for sharing and discussing their results in \cite{GMS} with us. Part of this work was completed while visiting the \emph{Random Geometry} programme at the Isaac Newton Institute. We wish to express our gratitude for the hospitality and the stimulating atmosphere.

\section{Background and setup}
\subsection{The critical FK model}\label{sec:critical-fk-model}
Recall that a \textbf{planar map} is a proper embedding of a
(multi) graph with $n$ edges in the plane which is viewed up to
orientation preserving homeomorphisms from the plane to itself. Let
$\bm_n$ be a map with $n$ edges and $\bt_n$ be the subgraph induced by a
 subset of its edges and \emph{all} of its vertices. We call the pair
$(\bm_n, \bt_n)$ a \textbf{submap decorated map}. Let $\bm_n^\dagger$ denote 
the \emph{dual map} of
$\bm_n$. Recall that the vertices of the dual map  correspond to the
faces of $\bm_n$ and two vertices in the dual map are adjacent if and
only if their corresponding faces are adjacent to a common edge in the
primal map. Every edge $e$ in the primal map corresponds to an edge $e^\dagger$ in the
dual map which joins the vertices corresponding to the two faces
adjacent to $e$. The
dual map $\bt_n^\dagger$ is the graph formed by the subset of edges $\{e^\dagger:
e\notin \bt_n\}$. We fix an oriented edge in the map $\bm_n$ and define it to 
be the
root edge.

\medskip Given a subgraph decorated map $(\bm_n,\bt_n)$, one can associate to it a set of
loops which form the interface between the two clusters. To define it
precisely, we consider a refinement of the map $\bm_n$ which is formed
by joining the dual vertices in every face of $\bm_n$ with the primal vertices
incident to that face. We call these edges \textbf{refinement
  edges}. Every
edge in $\bm_n$ corresponds to a quadrangle in its refinement formed by the union of the two triangles
incident to its two sides. In fact the refinement of $\bm_n$ is a
quadrangulation and this construction defines a
bijection between maps with $n$ edges and quadrangulations with $n$
faces.

There is an obvious one-one correspondence between the refinement edges 
and the
oriented edges in a map. Every oriented edge comes with a head and a
tail and a well defined triangle to its left. Simply match every
oriented edge with the refinement edge of the triangle to its left
which is incident to its tail. We call such an edge the refinement
edge corresponding to the oriented edge.

Given a subgraph decorated map $(\bm_n,\bt_n)$ define the map $(\bar
\bm_n, \bar \bt_n)$ to be formed by the
union of $\bt_n,\bt_n^\dagger$ and the
refinement edges. The root edge of $(\bar
\bm_n, \bar \bt_n)$ is the refinement edge corresponding to the root
edge in $\bm_n$ oriented towards the dual vertex. The \textbf{root
  triangle} is the triangle to the right of the root edge. It is easy to see that such a map is a
triangulation: every face in the refinement of $\bm_n$ is divided into
two triangles either by a primal edge in $\bt_n$ or a dual edge in
$\bt_n^\dagger$. Thus every triangle in $(\bar \bm_n, \bar \bt_n)$ is formed either by a
primal edge and two refinement edges or by a dual edge and two
refinement edges. For future reference, we call a triangle in $(\bar \bm_n, \bar \bt_n)$ with a
primal edge to be a \textbf{primal triangle} and that with a dual edge
to be a \textbf{dual triangle} (\cref{F:triangle}).

\medskip Finally we can define the interface as a subgraph of the dual map of the triangulation $(\bar \bm_n, \bar \bt_n)$. Simply
join together faces in the adjacent triangles which share a
common refinement edge. Clearly, such the interface is ``space
filling'' in the sense that every face in $(\bar \bm_n, \bar \bt_n)$ is traversed by
an interface. Also it is fairly straightforward to see that an
interface is a collection of {simple} cycles which we refer to as the
\textbf{loops} corresponding to the configuration $(\bm_n,\bt_n)$. Also such 
loops always have primal vertices one its one side and dual vertices on its 
other side. Also
every loop configuration corresponds to a unique  configuration $\bt_n$
and vice versa. Let $\ell(\bm_n,\bt_n)$ denote the number of loops
corresponding to a configuration $(\bm_n,\bt_n)$. The
critical FK model with parameter $q$ is a random configuration $(M_n,T_n)$ which follows
the law
\begin{equation}
  \label{eq:FKdef}
  \P((M_n,T_n)=(\bm_n,\bt_n)) \propto \sqrt{q}^{\ell(\bm_n,\bt_n)}
\end{equation}

The model makes sense for any $q \in [0,\infty)$ but we shall focus on
$q \in [0,4)$. It is also easy to see that the law of $(M_n,T_n)$
remains unchanged if we re-root the map at an independently and
uniformly chosen oriented edge (see, for example, \cite{UIPT1} for an argument).

\medskip Let $c(t_n)$ and $c(t_n^\dagger)$ denote the number of vertex clusters
of $t_n$ and $t_n^\dagger$. Recall that the loops form the interface 
between primal and dual vertex clusters. From this, it is clear that
$\ell(\bm_n,\bt_n) = c(t_n)+c(t_n^\dagger)-1$. Let $v(G), e(G)$ denote the
number of vertices and edges in a graph $G$. An application of Euler's formula
shows that
\begin{equation}
  \label{eq:1}
  \P(\bm_n) = \frac1Z (\sqrt{q})^{-v(\bm_n)} \sum_{G \subset \bm_n} \sqrt{q}^{e(G)}q^{c(G)}.
\end{equation}
where $Z$ denotes the partition function. It is easy to conclude
from this
that the model is self dual and hence critical. Note that \eqref{eq:1}
corresponds to the Fortuin-Kasteleyn random cluster model which is in
turn is equivalent to the $q$-state Potts model on  
maps with $n$ edges (see \cite{BKW}). 

\subsection{Sheffield's bijection} \label{sec:HC bij}

We briefly recall the Hamburger--Cheeseburger bijection due to
Sheffield (see also related works by Mullin \cite{mullin67} and
Bernardi \cite{bernardi07,bernardi08}).

\begin{figure}
\begin{center}
\includegraphics[width=.1\textwidth]{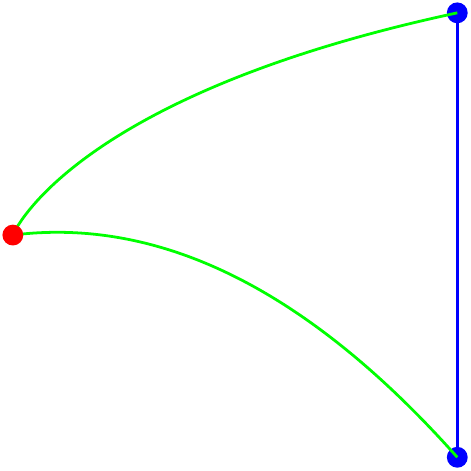} \quad \quad 
\quad 
\includegraphics[width=.1\textwidth]{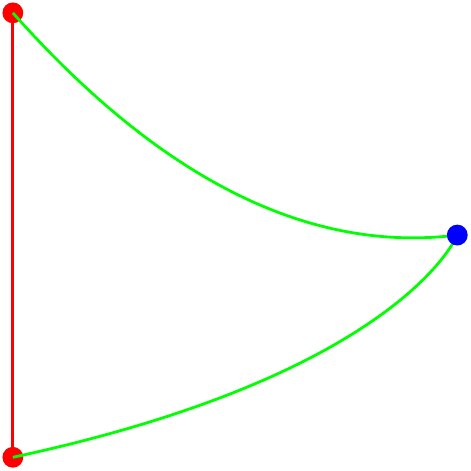}\quad \quad 
\includegraphics[width=.2\textwidth]{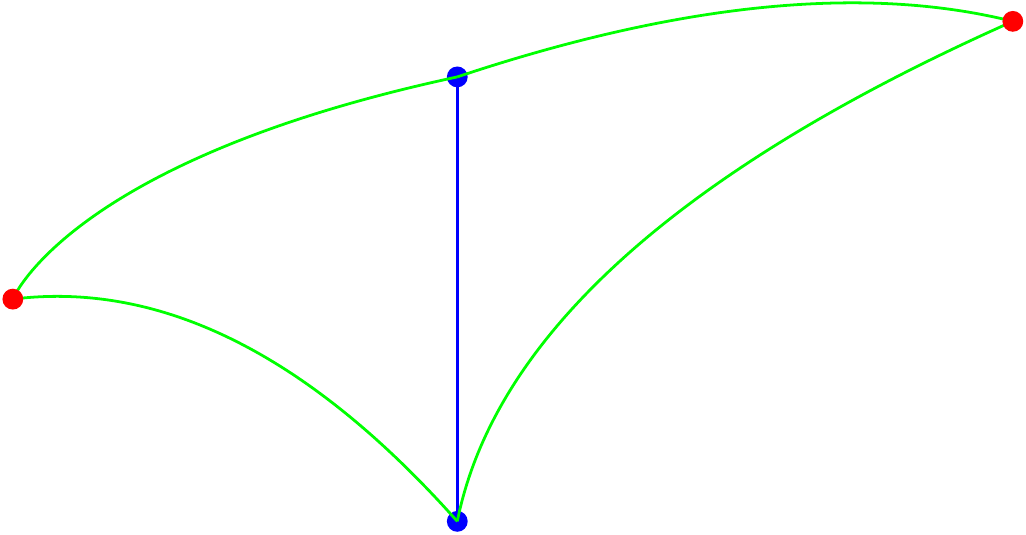}
\caption{Refined or green edges split the map and its dual into primal and dual triangles. Each primal triangle sits opposite another primal triangle, resulting in a primal quadrangle as above. }
\label{F:triangle}
\end{center}
\end{figure}

\medskip Recall that the refinement edges split the map into triangles which can be of only two types: a primal triangle (meaning two green or refined edges and one primal edge) or a dual triangle (meaning two green or refined edges and one dual edge). For ease of reference primal triangles will be associated to hamburgers, and dual triangles to cheeseburgers. Now consider the primal edge in a primal triangle; the triangle opposite that edge is then obviously a primal triangle too. Hence it is better to think of the map as being split into quadrangles where one diagonal is primal or dual (see \cref{F:triangle}).

\medskip We will reveal the map, triangle by triangle, by exploring it with a 
path which visits every triangle once (hence the word ``space-filling"). We will 
keep track of the first time that the path enters a given quadrangle by saying 
that {either a hamburger or a cheeseburger is produced, depending on whether the quadrangle is primal or dual.} 
Later on, when the path comes back to the quadrangle for the second 
and last time, we will say that the burger has been eaten. We will use the 
letters $\ah, \ac$ to indicate that a hamburger or cheeseburger has been 
produced and we will use the letters $\aH, \aC$ to indicate that a burger 
has been eaten (or \emph{ordered} and eaten immediately). So in this description we will have one letter for every triangle.

\medskip It remains to specify in what order are the triangles visited, or 
equivalently to describe the space-filling path. In the case where the 
decoration $\bt_n$ consists of a single spanning tree (corresponding to $q=0$ as 
we will see later) the path is simply the contour path going around the 
tree. Hence in that case the map is entirely described by a sequence of $2n$ 
letters in the alphabet $\{\ah, \ac, \aH, \aC\}$. 

\medskip We now describe the situation when $\bt_n$ is arbitrary, which is more delicate. The idea is that the space-filing path starts to go around the component of the root edge, i.e. explores the loop of the root edge, call it $L_0$. However, we also need to explore the rest of the map. To do this, consider the last time that $L_0$ is adjacent to some triangle in the complement of $L_0$, where by complement we mean the set of triangles which do not intersect $L_0$. (Typically, this time will be the time when we are about to close the loop $L_0$). At that time we continue the exploration as if we had \textbf{flipped} the diagonal of the corresponding quadrangle. This has the effect the exploration path now visits two loops. We can now iterate this procedure. A moment of thought shows that this results in a space-filling path which visit every quadrangle exactly twice, going around some virtual tree which is not needed for what follows. We record a flipping event by the symbol $\aF$. More precisely, we associate to the decorated map $(\bm_n, \bt_n)$ a list of $2n$ symbols $(X_i)_{1\le i \le 2n}$ taking values in the alphabet $\Theta = \{ \ah, \ac, \aH, \aC, \aF\}$. For each triangle visited by the space-filling exploration path we get a symbol in $\Theta$ defined as before if there was no flipping, and we use the symbol $\aF$ the second time the path visit a flipped quadrangle.

 It is not obvious but true that this list of symbols completely
 characterises the decorated map $(\bm_n, \bt_n)$. Moreover, observe
 that each loop corresponds to a symbol $\aF$ (except the loop through the 
root).
\subsection{Inventory accumulation}

We now explain how to reverse the bijection. One can interpret an element in $\{\ah,\ac,\aH,\aC\}^{2n}$ as a
last-in, first-out inventory accumulation process in a burger factory with two types of products:
hamburgers and cheeseburgers. Think of a sequence of events
occurring per unit time in
which either a burger is produced (either ham or cheese) or there is
an order of a burger (either ham or cheese). The burgers are put in a
single \textbf{stack} and every time there is an order of a certain
type of burger, the freshest burger in the stack of the corresponding
type is removed. The symbol $\ah $ (resp. $\ac$) corresponds to a ham
(resp. cheese) burger production and the symbol $\aH$ (resp. $\aC$)
corresponds to a ham (resp. cheese) burger order. 



\begin{figure}
  \centering
  \includegraphics[scale = 0.8]{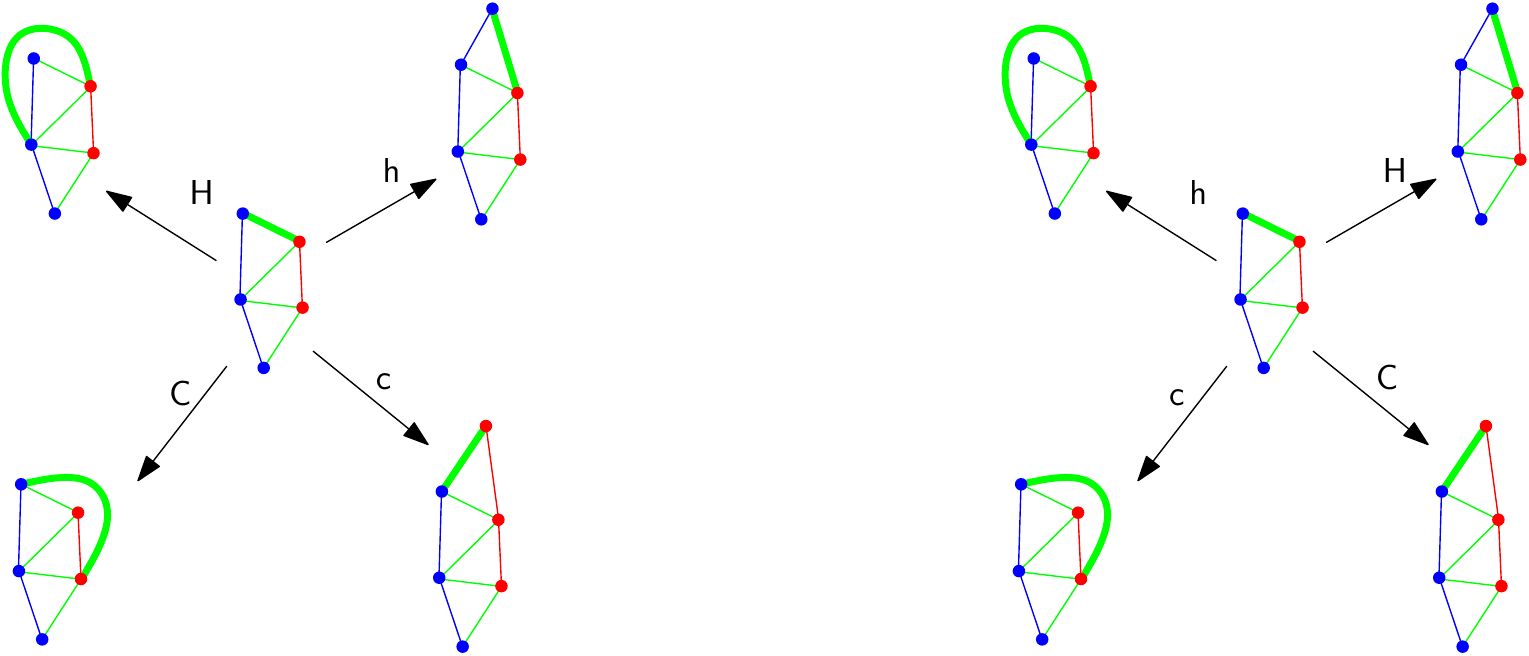}
  \caption{From symbols to map. The current position of the interface (or last discovered refined edge) is indicated with a bold line. 
  Left: reading the word sequence from left to right or \emph{into the future}. The map in the center is formed from the 
symbol sequence $\ac\ah\ac$. 
Right: The corresponding operation when we go from right 
to left (or into the \emph{past}). The map in the center now corresponds to the 
symbol sequence $\aC\aH\aC$.}
  \label{fig:map_from_word}
\end{figure}

Reversing the procedure when there is no $\aF$ symbol is pretty obvious (see e.g. \cref{fig:map_from_word}). 
 So we discuss the general case straight away. 
 The inventory interpretation of the symbol $\aF$ is the following:
this corresponds to a customer
demanding the freshest or the topmost burger in the stack irrespective
of the type. In particular, whether an $\aF$ symbol corresponds to a
hamburger or a cheeseburger order depends on the topmost burger type
at the time of the order. Thus overall, we can think of the inventory
process as a sequence of symbols in $\Theta$ with the following
reduction rules
\begin{itemize}
\item $\ac \aC = \ac \aF = \ah \aH = \ah \aF = \emptyset$,
\item $\ac\aH = \aH \ac$ and $\ah \aC = \aC \ah$.
\end{itemize}
Given a sequence of symbols $X$, we denote by $\bar X$ the reduced
word formed via the above reduction rule.

Given a sequence $X$ of symbols from $\Theta$, such that $\bar X =
\emptyset$, we can construct a decorated map $(\bm_n,\bt_n)$
as follows. First convert all the $\aF$ symbols to either a $\aH$ or
an $\aC$ symbol depending on its order type. Then construct a spanning
tree decorated map as is described above
(\cref{fig:map_from_word}). The condition $\bar X = \emptyset$ ensures
that we can do this. To obtain the loops, simply switch the type of
every quadrangle which has one of the triangles corresponding to an $\aF$
symbol. That is, if a quadrangle formed by primal triangles has one of
its triangles coming from an $\aF$ symbol, then replace the primal map
edge in that quadrangle by the corresponding dual edge and vice
versa. The interface is now divided into several loops and the number
of loops is exactly one more than the number of $\aF$ symbols.

\paragraph{Generating FK-weighted maps.} Fix $p \in [0,1/2)$. Let $X_1,\ldots, X_n$ be i.i.d.\ with the
following law

\begin{equation}
  \label{eq:3}
  \P( \ac) = \P( \ah ) = \frac14, \P( \aC ) = \P ( \aH ) = \frac{1-p}{4}, \P( \aF ) = \frac p2.
\end{equation} 
conditioned on $\overline{X_1,\ldots, X_n} = \emptyset$. 

Let $(\bm_n, \bt_n)$ be the random associated decorated map as above. Then observe that since $n$ hamburgers and cheeseburgers must be produced, and since $\#\aH + \#\aC = n - \# \aF$,
\begin{align}
  \label{eq:2}
  \P((\bm_n,\bt_n)) & = \left(\frac14\right)^n  \left( \frac{1-p}4\right)^{\#\aH + \#\aC} \left(\frac{p}{2}\right)^{\#\aF}\nonumber  \\
  & \propto \left(\frac{2p}{1-p}\right)^{\# \aF} = \left(\frac{2p}{1-p}\right)^{\# \ell(\bm_n,\bt_n)-1} 
\end{align}
Thus we see that $(\bm_n,\bt_n)$ is a realisation of the critical FK-weighted cluster random map model with $\sqrt{q}
= 2p/(1-p)$. Notice that $p \in [0,1/2)$ corresponds to $q =
[0,4)$. From now on we fix the value of $p$ and $q$ in this regime. (Recall that $q=4$ is believed to be a critical value for many properties of the map).

\subsection{Local limits and the geometry of loops}\label{sec:local}
The following theorem due to Sheffield and made more precise later by Chen
\cite{She11,chen} shows that the decorated map $(M_n,T_n)$ has a local limit
as $n \to \infty$ in the local topology. Roughly two maps
are close in the local topology if the finite maps near a large
neighbourhood of the root are isomorphic as maps (see \cite{BeSc} for
a precise definition).

\begin{thm}[\cite{She11,chen}]
  \label{thm:Local}
Fix $p \in [0,1)$. We have
\[
(M_n,T_n) \xrightarrow[n \to \infty]{(d)}(M,T)
\]
in the local topology. 
\end{thm}

Furthermore, $(M,T)$ can be described by applying the obvious infinite version of Sheffield's bijection to the  bi-infinite i.i.d.\ sequence of symbols with
law given by \eqref{eq:3}.

The idea behind the proof of \cref{thm:Local} is the following. Let
$X_1,\ldots,X_{2n}$ be i.i.d.\ with law given by \eqref{eq:3}
conditioned on
$\overline{X_1\ldots X_{2n}} = \emptyset$. It is shown in
\cite{She11,chen} that the probability of $\overline{X_1\ldots X_{2n}} =
\emptyset$ decays sub exponentially. Using Cramer's rule one can deduce
that locally the symbols around a uniformly selected symbol from
$\{X_i\}_{1\le i \le n}$ converge to a bi-infinite i.i.d.\ sequence
$\{X_i\}_{i \in \Z}$ in law. The proof is now completed by arguing that the correspondence
between the finite maps and the symbols is continuous in the local
topology.

Notice that uniformly selecting a symbol corresponds to selecting a
uniform triangle in $(\bar M_n,\bar T_n)$ which in turn corresponds to
a unique refinement edge which in turn corresponds to a unique
oriented edge in $M_n$. Because of the above interpretation and the invariance under
re-rooting, one can think of the triangle corresponding to $X_0$ as
the root triangle in $(M,T)$. 

One important thing that comes out of the proof is that every symbol
in the i.i.d.\ sequence $\{X_i\}_{i \in \Z}$ has an almost sure unique
\textbf{match}, meaning that every order is fulfilled and every burger is
consumed with probability $1$. Let $\varphi(i)$ denote the match of
the $i$th symbol. Notice that $\varphi:\Z \mapsto \Z$ defines an
involution on the integers.

The goal of this section is to explain the connection between the
geometry of the loops in the infinite map $(M,T)$ and the bi-infinite sequence $\{X_i\}_{i \in \Z}$ of symbols with law
given by \eqref{eq:3}. For this, we describe an equivalent procedure
to explore the map associated to a given sequence, triangle by triangle in the refined map $(\bar
M,\bar T)$.  (This is again defined in the same way as its finite
counterpart: it is formed by the subgraph $T$, its dual $T^\dagger$
and the refinement edges.)

\paragraph{Loops, words and envelopes.} Recall that in the infinite (or whole-plane) decorated refined map $(\bar M,\bar T)$, each loop is encoded by a unique $\aF$ symbol in the bi-infinite sequence of symbols $(X_i)_{i \in \Z}$, and vice-versa. Suppose $X_i = \aF$ for some $i \in \Z$, and consider the word $W = X_{\ph(i) } \ldots X_i$ and the reduced word $\cR = \bar W$ (recall that $\ph(i)$ is a.s. finite). Observe that $\cR$ is necessarily of the form $\aH \ldots \aH$ or of the form $\aC\ldots \aC$ depending on whether $X_{\ph(i)} = \ac$ or $\ah$, respectively. These symbols can appear any number of times, including zero if $\cR = \emptyset$. 

A moment of thought shows therefore that $W$ encodes a decorated
submap of $(\bar M, \bar T)$ which we call the envelope of $X_i$,
denoted by $\ae(i)$ or sometimes $\ae(X_i)$ with an abuse of notation.
Furthermore, this map is finite and simply connected. Assume without
loss of generality that $\cR$ contains only $\aH$ symbols. Then the
boundary of this map consists a connected arc of $|\cR|  $ primal
edges and two green (refined) edges (see \cref{F:env}).  Note also
that this map depends only on the symbols $(X_j)_{\varphi(i) \le j \le
  i}$ (i.e., $W$ do not contain any $\aF$ symbol whose match is
outside $W$).

\medskip The complement of the triangles
corresponding to a loop in $(\bar M,\bar T)$ consists of one infinite
component and several finite components (there are several components if the loop forms fjords). Recall that the loop is a
simple closed cycle in the dual of the refined map, hence it divides
the plane (for any proper embedding) into an \emph{inside} component
and an \emph{outside} component.

\begin{defn}\label{def:area}
  Given a loop in the map $(\bar M,\bar T)$, the \textbf{interior of the
  loop}  is the portion of the map corresponding to
the triangles in the finite
component of its complement and lying completely inside the
loop. The rest of the triangles lie in the \textbf{exterior} of the loop. The \textbf{length of the loop} is the number of triangles
corresponding to the vertices (in the dual refined map) in the loop,
or equivalently, the number of triangles that the loop goes
through. The \textbf{area inside the
  loop} is the number of triangles in its interior plus the length of
the loop.
\end{defn}

\begin{figure}[h]
\begin{center}
\includegraphics[scale=.4]{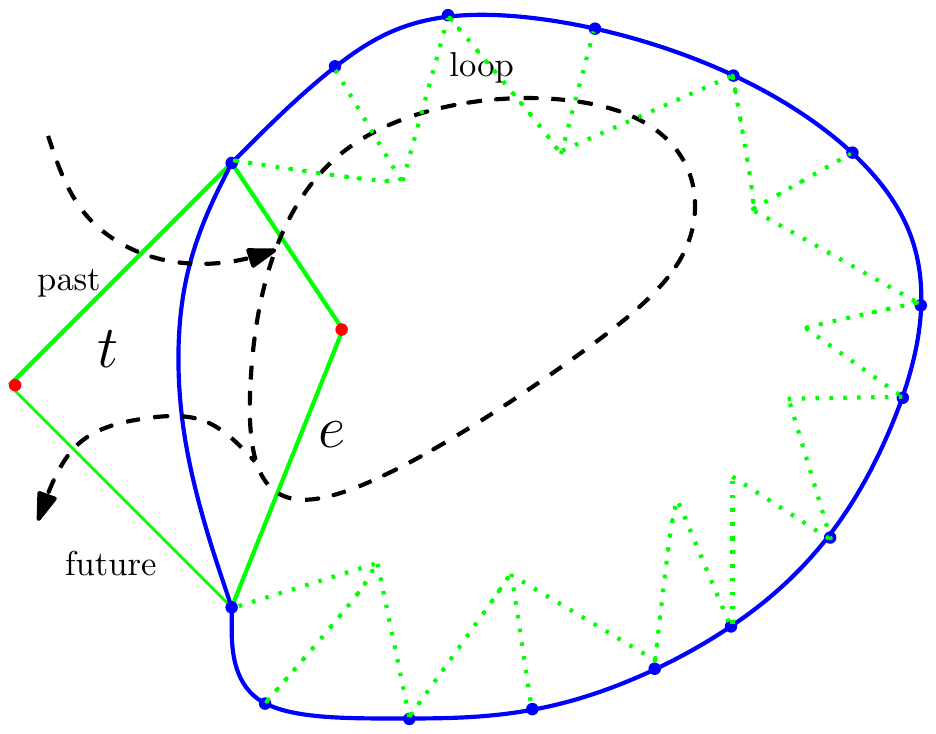}\quad \quad \quad \quad \quad
\includegraphics[scale=.4]{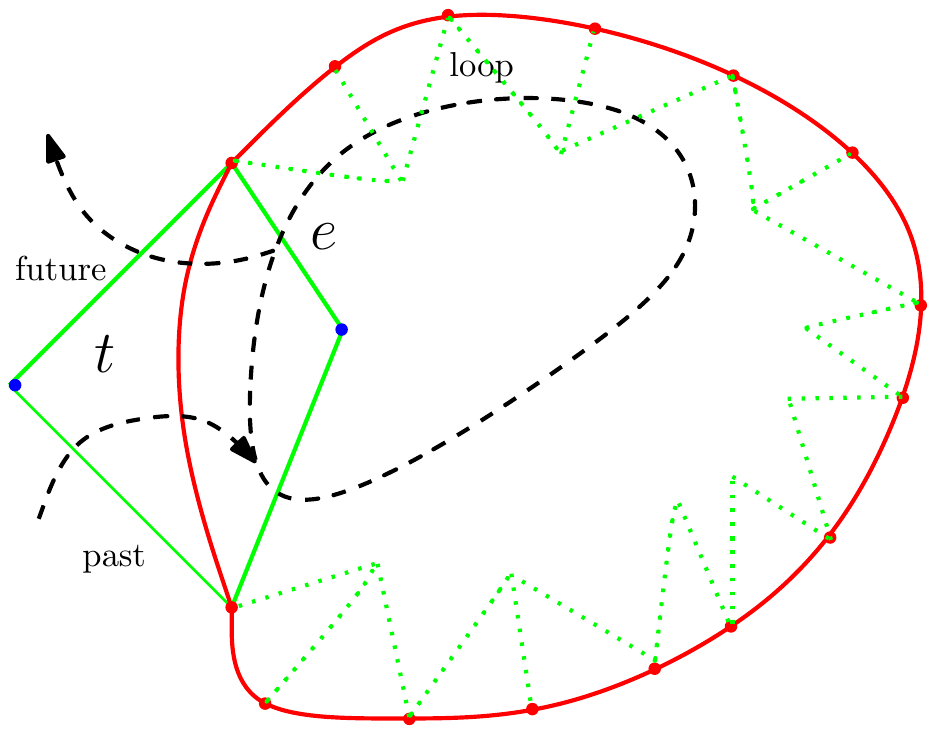}
\caption{The envelope of an $\aF$ symbol matched to a $\ac$. The green 
quadrangle corresponds to the $\aF$ and its match. All the other blue edges on 
the boundary of the envelope correspond to the symbols $\aH$ in the reduced 
word 
$\cR$. Note that not all triangles on the boundary of the envelope are part of 
the loop itself. Right: the corresponding map if the $\aF$ symbol is matched 
with an $\ah$.}
\label{F:env}
\end{center}
\end{figure}
\medskip We now describe an explicit exploration procedure of an envelope, 
starting from its $\aF$ symbol, and exploring towards the past. 

\paragraph{Exploration into the past for an envelope.} We start with a single 
edge $e$ and we explore the symbols strictly to the left of the $\aF$ symbol. 
At every step we reveal a part of the map incident to an edge which we explore.

\begin{enumerate}\label{enum}

 \item If the symbol is a $\aC,\aH$ or a $\ac,\ah$ which is not the match of 
the $\aF$, then we glue a single triangle to the edge we explore as in the 
right hand side of \cref{fig:map_from_word}.   

\item If the symbol is an $\aF$, we explore its envelope and glue the 
corresponding map as explained above (see \cref{F:env}). The refined edge 
corresponding to the ``future'' in \cref{F:env} is identified with the edge 
we are exploring and the edge corresponding to the ``past'' is the edge we 
explore next.

\item If the symbol is a $\ac$ or $\ah$ and is a match of the $\aF$ symbol we 
started with, we finish the exploration as follows. Notice that in this 
situation, if the symbol is a $\ac$ (or $\ah$) then the edge we explore is 
incident to $e$ via a dual (or primal) vertex. We now glue a primal (or dual) 
quadrangle with two of its adjacent refined edges identified with $e$ and the 
edge we explore. This step corresponds to adding the quadrangle with solid 
lines in \cref{F:env}. 
\end{enumerate}
\begin{remark}
 \label{rem:extension}
We remark that it is possible to continue the exploration procedure above for 
the whole infinite word to the left of $X_0$. The only added subtlety is that 
some productions have a match to the right of $0$ and hence remain uneaten. 
The whole exploration thus produces a half planar map with boundary formed 
by these uneaten productions. However this information does not reveal all the 
decorations in the boundary since some of the boundary triangles might be 
matched by an $\aF$ to the right of $X_0$. 
\end{remark}
We now explain how to extract information about the
length and area of the loop given the symbols in an envelope. A
preliminary observation is that the envelopes are nested. More
precisely, if $X_i = \aF$ and $X_j = \aF$ for some $\varphi(i ) <j
<i$. Then $\ae(X_j) \subset \ae(X_i)$. To see this, observe that a
positive number of burgers are
 produced between $X_{\varphi(i)}$ and $X_i$ and hence one of them
 must match $X_j$. Since it cannot be $X_{\varphi(i)}$ by definition,
 $\varphi(j) > \varphi(i)$.

 If we define a partial order
among the envelopes strictly contained in $\ae (X_i)$ then there exist maximal
elements which we call \textbf{maximal envelopes} in $\ae(X_i)$. 
\begin{lemma}\label{prop:loop+area}
Suppose $X_i = \aF$, and let $L$ be the corresponding loop. Then the following holds.
  \begin{itemize}
\item The boundary of $\ae(X_i)$, that is the triangles in $\ae(X_i)$
  which are adjacent to triangles in the complement of $\ae(X_i)$, consists of
  triangles in the reduced word $\overline {X_{\ph(i)} \ldots
    X_i}$, plus one extra triangle (corresponding to $t$ in \cref{F:env}). For an $\ah\aF$ loop, the boundary consists of dual
  triangles corresponding to $\aC$  symbols.  An identical statement
  holds for an $\ac\aF$ loop with dual replaced by primal. 

\item Let $\mathcal M $ denote the union of the maximal envelopes in
  $\ae(X_i)$ and let $m$ denote the number of maximal envelopes in $\ae(X_i)$. Then the length of $L$ is $m$ plus the number of
  triangles in $\ae(X_i) \setminus \cM$ minus $1$.
\item All the envelopes in $\cM$ of type opposite (resp. same) to that of $L$ belong to the interior (resp. exterior) of $L$.
  \end{itemize}
 
\end{lemma}
\begin{proof}
 The boundary of $\ae(X_i)$ is formed of symbols that are going to be matched by
 symbols outside $ [ \ph(i), i ]$. Thus by definition, the boundary consists of the triangles associated with the
 reduced word $\overline {X_{\ph(i)} \ldots X_i}$. Also for an
 $\ah\aF$ loop, the boundary consists of $\aC$ symbols only since if
 there was an $\aH$ symbol, it would have been a match of
 $\varphi(i)$. An identical argument holds for a $\ac\aF$ loop.

For the second assertion, suppose we start the exploration procedure for a loop 
going into the past as described above. For steps as in item $1$, it is clear 
that we add a single triangle to the loop. For steps as in item $2$, i.e. when 
we reveal the map corresponding to a maximal envelope $E$, we also add a single 
triangle to the loop. Indeed an envelope consists of a single triangle $t$ glued 
to a map bounded by a cycle of either primal or dual edges (see \cref{F:env}). 
If we iteratively explore $E$, $t$ is part of the quadrangle we add in 
step $3$ of the above exploration and it is the triangle $t$ which is added to 
the loop. For steps as in item $3$, we also add one triangle to the loop. This 
concludes the proof of the second assertion. 


Clearly, the triangles corresponding to a loop has primal vertices on
one side and dual vertices on the other side of the loop. Suppose $X_i$ is $\ah\aF$ type. Then, as for any such loop, it has dual (or $\aC$) vertices adjacent to its
exterior.  
For the same reason, every $\ah\aF$ type maximal
envelope in $\ae(X_i)$ must have dual (or $\aC$) vertices adjacent to its exterior. 
None of its triangles belong to the loop by the second assertion, and it is adjacent to $L$. So the only possibility is that it lies in its exterior.
The other case is similar, so the last assertion is proved.
\end{proof}
\section{Preliminary lemmas}

\subsection{Forward-backward equivalence}\label{sec:fbequiv}

In this section, we reduce the question of computing critical exponents on the decorated map to a more tractable question on certain functionals of the Hamburger Cheeseburger sequence coming from Sheffield's bijection. This reduction involves elementary but delicate identities and probabilistic estimates which need to be done carefully. By doing so we describe the length and area by quantities which have a more transparent random walk interpretation
and that we will be able to estimate in \cref{S:rw_estimate}. 



Modulo these estimates, we complete the proof of
\cref{T:typ} at the end of \cref{sec:fbequiv}.  From now on throughout the rest of the paper, we fix the following notations:
\begin{defn}\label{symbols}
Fix $p
  \in (0,1/2)$. Define 
$$
\theta_0 = 2\arctan\left(\frac{1}{\sqrt{1-2p}}\right); \quad p_0 =
\frac{\pi}{2\theta_0} = \frac {\kappa'} 8.
$$
\end{defn}
Note that the value of $p_0$ is identical to the one in \eqref{p0}
(after applying simple trigonometric formulae). Also assume throughout
in what follows that
$\{X_k\}_{k \in \Z}$ is an i.i.d.\ sequence given by \eqref{eq:3}.



For any $k \in \Z$, we define a \textbf{burger stack} at time $k$ to
be $((X_j)_{j \le k}: \varphi(j) >k)$ endowed with the natural
order it inherits from $\{X_k\}_{k \in \Z}$. The maximal element in a burger
stack is called the burger or symbol at the top of the stack. It is
possible to see that almost surely burger stack
at time $k$ contains infinite elements almost surely for any $k \in
\Z$ (see \cite{She11}).

Define $
T = \varphi(0)$, and
let 
$
E = 
\{ X_{\varphi(0)} = \aF\}
$. 
 Let $J_T =
\overline{X_0\ldots X_T}$. Let $|J_T|$ denote the number of symbols
in $\overline{X_0\ldots X_T}$. Let $\mathcal S_k$ denote the burger stack
at time $k$. Let $\P^{s}$ denote the probability measure $\P$
conditioned on $\cS_0 = s$. Note that conditioning on the whole past $\{X_j\}_{j \le k} $ at a given time $k$ is equivalent to conditioning just on the burger stack $s$ at that time. 

\begin{thm}\label{thm:exit}
Let $T,E,J_T$ be as above and $p_0 = \kappa'/8$ be as in \cref{symbols}. Fix $\ve >0$. 
There
exist positive constants $c=c(\ve),C=C(\ve)$ such that for all $n \ge 1, m 
\ge n (\log n)^3$, for any burger stack $s$,
\begin{enumerate}[{(}i{)}]

\item \label{item1}$\frac{c n^{2p_0}}{n^{1+\ve}m^{4p_0+\ve}} \le 
\P^s(T>m^2,|J_T| = n,E) \le 
\frac{Cn^{2p_0}}{n^{1-\ve}m^{4p_0-\ve}}$,

\item \label{item2}$\frac{c}{n^{2p_0+1+\ve}} \le \P^s(|J_T| = n,E) \le 
\frac{C}{n^{2p_0+1-\ve}}$,

\item \label{item3}$\frac{c}{m^{2p_0 + \ve}} \le \P^s(T > m^2 \cap E ) \le
\frac{C}{m^{2p_0 - \ve}}$,

\item \label{item4}$c\left(\frac{n^{4p_0-\ve}}{m^{4p_0+\ve}}\right) \le 
\P^{s}(T>m^2 \big| |J_T |= n \cap E ) \le 
C\left(\frac{n^{4p_0+\ve}}{m^{4p_0-\ve}}\right)$,


\item \label{mainitem} For any $p\in  (0,2p_0-\ve)$,

\begin{equation*}
 c n^{2p-2\ve} \le \E^s(T^p \big| |J_T| = n\cap E) \le C n^{2p+2\ve}.
\end{equation*}
\end{enumerate}
In particular all these bounds are independent of the conditioning on
$\mathcal S_0 =s$.
\end{thm}



\begin{remark}
 A finer asymptotics than (iii) above is obtained in
 \cite[Proposition 5.1]{GMS}. More precisely, it is proved that
 $\P(T>n\cap E) $ is regularly varying with index $p_0 =
 \kappa'/8$.
\end{remark}

Let us admit \cref{thm:exit} for now and let us check how this implies \cref{T:typ}. To do this we need to relate $T, E$ and $J_T$ to observables on the map.
We now check some useful invariance properties which 
use the fact that there are various equivalent ways of defining a typical loop.

\begin{prop}\label{prop:typical_law}
The following random finite words have the same law. 
\begin{enumerate}[{(}i{)}]
\item The envelope of the first $\aF$ to the left of $X_0$. That is
  $\ae(X_i)$ where $i = \max\{j \le 0,X_j = \aF\}$.\label{typ1}
\item The envelope of the first $\ac$ or $\ah$ to the right of $0$
  matched with an $ \aF$. That is $\ae(\varphi(\Sigma))$ where $ \Sigma = \min 
\{j\ge 0:X_j \in
  \{\ac,\ah\},X_{\ph(j)} = \aF\}$.\label{typ2}
\item The envelope of $X_0 $ conditioned on
  $X_0$ being an $ \aF$.\label{typ3}
\item The envelope of $X_{\ph(0)}$, conditioned on $X_{\ph(0)} = 
\aF$.\label{typ4}
\end{enumerate}
Furthermore this is the limit law as $n \to \infty$ for the envelope of a $\aF$ taken
uniformly at random from an i.i.d. sequence $X_1,\ldots, X_{2n}$
distributed as in \eqref{eq:3} and conditioned on $\overline{X_1\ldots X_{2n}} = \emptyset$.
\end{prop}
\begin{proof}
Let $\cW$ be the set of finite words $\{w_0,\ldots,w_n, w_n = \aF, n = \ph(0)\}_{n \ge 0}$ of any length, that end with an
$\aF$ and start by its match. For $w=(w_0,\ldots,w_n) \in \cW$, let $p(w) =
\prod_{i=0}^{n}\P(X = w_i)$. Let $p_{\cW}(w) = p(w)/Z$ where $Z =
\sum_{w \in \cW} p(w)$. Clearly, $Z = \P(X=\aF) = p/2$, since $\sqcup_{w \in 
\cW} \{X_{-n} = w_0 , \ldots, X_0 = w_n \} = \{ X_0 = \aF\}$ and these events 
are disjoint.

Note that the word in the \cref{typ3} has a law given by $\P( e( X_0) = w | X_0 
= \aF) = (2/p) \prod_{i=0}^n \P(X = w_i) = p_{\cW}(w)$. This is also true of 
the word in the \cref{typ1}, since the law to the sequence to the left of
 the first $\aF$ left of $0$ is still i.i.d. 
 
 Similarly, for the word in the \cref{typ4}, conditioning on $X_{\ph(0)} = \aF$ 
is the same thing as conditioning on $\ae(X_{\ph(0)}) \in \cW$ hence it follows 
that the random word has law $p_{\cW}$ too. This then immediately implies the 
result in the \cref{typ2}, 
 since conditioned on the $k$th burger produced after 
time $0$ to be the first one eaten by an $\aF$, the envelope of the $k$th 
burger produced has law $p_{\cW}$, independently of $k$.
 

The final assertion is a consequence of polynomial decay of empty reduced word as
described in \cite{She11,chen} which we provide for
completeness. For $w \in \cW$ be a word with $k$ symbols. Let $N_w$ be the number of $\aF$
symbols in $X_1,\ldots,X_n$ such that its envelope is given by $w$. Let
$N_{\aF}$ denote the number of $\aF$ symbols  in $X_1,\ldots,X_n$. We
can treat both $N_{\aF}$ and $N_w$ as empirical measure of states $w$
and $\aF$ of certain Markov chains of length $n- k$ and $n$ respectively. By
Sanov's theorem,
\[
\P(|\frac {N_w}{n} - p(w)| >\ve) \le ce^{-c'n}\quad ; \quad \P(|\frac{N_{\aF}}{n} -
p/2|>\ve) \le ce^{-c'n}
\]
Since $\P(\overline{X_1 \ldots X_{2n}}) = \emptyset) =n^{-1-\kappa/4
  +o_n(1)}$ \cite{gms2}, our result follows. See for example
\cite{chen} for more precise treatment of similar arguments. 
%
\end{proof}
Let $(M_n,T_n)$ be as in \cref{eq:FKdef} and let $L_n$ be a uniformly
picked loop from it. One can extend the definition of length, area, exterior and interior in \cref{def:area} to finite maps by adding the convention that the exterior of a loop is the component of the complement containing the root.
 (If the loop intersects the root edge, we
define the interior to be empty.) Let $\cL_n$ be the submap of $(\bar
M_n,\bar T_n)$ formed by the triangles corresponding to the loop $L_n$
and the triangles in its interior. Recall that by definition, the length of the loop, denoted $\len(L_n)$ is the number
of triangles in 
$(\bar M_n,\bar T_n)$ present in the loop and the
area $\Area(L_n)$ is the number of triangles in $\cL_n$, that is, the number of triangles in the interior of the loop
plus $\len(L_n)$.

\begin{prop}\label{lem:typ_area_loop}
 The number of triangles in $\cL_n$ is tight and $\cL_n$ converges to
 a finite map $\cL$. The submap corresponding to triangles in $L_n$ converges to
 a map $L$. Also
 \begin{itemize}
 \item $\len(L_n) \xrightarrow[]{n \to \infty} \len(L)$
\item  $\Area(L_n) \xrightarrow[]{n \to \infty} \Area(L)$
 \end{itemize}
where $\len(L)$ is the number of triangles in $L$ and $\Area(L)$ is
the number of triangles in $\cL$. Further the law of $\len(L)$ and
$\Area(L)$ can be described as follows. Take an i.i.d.\ sequence
$\{X_i\}_{i \in \Z}$ as in \cref{eq:3} and condition on $X_0 =
\aF$. Then the map corresponding to $\ae(X_0)$ has the same law as
$\cL$. Thus the law of $\len(L)$ and $\Area(L)$ can be described in
the way prescribed by \cref{prop:loop+area}.
\end{prop}
\begin{proof}
  Notice that there is a one to one correspondence between the number of
$\aF$ symbols in the finite word corresponding to $(M_n,T_n)$
except there is one extra loop. But since the number of $\aF$
symbols in the finite word converges to infinity, the probability that
we pick this extra loop converges to $0$. The rest follows from the last statement in \cref{prop:typical_law,prop:loop+area}.
\end{proof}

We now proceed to the proof of \cref{T:typ}. 
We compute each exponent separately. In this proof we will make use of certain 
standard type exponent computations for i.i.d.\ heavy tailed random variables. 
For clarity, we have collected these lemmas in \cref{A}.

\paragraph{Proof of length exponent in \cref{T:typ}.}  We see from \cref{lem:typ_area_loop} that it is enough to condition on $X_0 = \aF$
  and look at the length of the loop and area of the envelope
  $\ae(X_0)$ as defined in \cref{def:area}. We borrow the notations from
  \cref{lem:typ_area_loop}. We see from the second item of
  \cref{prop:loop+area}, to get a handle on $\len(L)$, we need to
  control the number of  maximal envelopes and the number of triangles
  not in maximal envelopes inside $\ae(X_0)$. To do this, we define a sequence
  $(c_n,h_n)_{n \ge 1}$ using the \emph{exploration into the past for an 
envelope} as described in \cref{sec:local} and keeping track of the number of 
$\aC$ and $\aH$ in the reduced word. Let $(c_0,h_0 ) =
  (0,0)$. Suppose we have performed $n$ steps of the exploration and defined 
$c_n,h_n$ and in this process, we have revealed triangles corresponding to 
symbols $(X_{-m}, \ldots, X_0)$.
We inductively define the following.
\begin{itemize}
\item If $X_{-m-1}$ is a $\aC$  (resp. $\aH$), define $(c_{n+1},h_{n+1}) = 
(c_n,h_n)
  + (1,0)$ (resp. $(c_n,h_n)
  + (0,1)$).
\item If $X_{-m-1}$ a $\ac$ (resp. $\ah$), $(c_{n+1},h_{n+1}) = (c_n,h_n)
  + (-1,0)$ (resp. $(c_n,h_n)
  + (0,-1)$).
\item If $X_{-m-1}$ is $\aF$, then we explore $X_{-m-2},X_{-m-3}\ldots$ until we 
find the
  match of $X_{-m-1}$. Notice that the reduced word
  $\cR_{n+1}=\overline{X_{\ph( -m-1)} \ldots X_{-m-1}}$ is either of the form 
$\aC \aC \ldots \aC $ or
  $\aH \aH \ldots \aH$ depending on whether the match of the $\aF$
  is a $\ah$ or $\ac$ respectively. Either happens with equal
  probability by symmetry. Let $|\cR_{n+1}|$ denote
  the number of symbols in the reduced word $\cR_{n+1}$. If $\cR_{n+1}$ consists 
of $\aH$ symbols,
  define $(c_{n+1},h_{n+1}) = (c_n,h_n)
  + (0,|\cR_{n+1}|)$. Otherwise, if $\cR_{n+1}$ consists of $\aC$ symbols 
define $(c_{n+1},h_{n+1}) = (c_n,h_n)
  + (|\cR_{n+1}|,0)$.
\end{itemize}
For future reference, we call this exploration procedure the \textbf{reduced 
walk}.


%

Observe that the time $\ph(0)$ where we find the match of 0 in the reduced walk
is precisely the time $n$ when the process $(c_n,h_n)$
leaves the first quadrant, i.e.,  $\tau :=
\inf\{k: c_k \wedge h_k<0\}$. This is because $\tau$ is the first step when
$\overline{X_{-\tau}\ldots X_{-1}}$ consists of a $\ac$ or $\ah$
symbol followed by a (possibly empty) sequence of burger orders of the
opposite type and hence the $\ac$ or $\ah$ produced is the match of
$\aF$ at $X_0$. Also from second item of  \cref{prop:loop+area}, $\tau$ 
is exactly the number of triangles in the loop (as exploring the envelope of each $\aF$ corresponds to removing the maximal envelopes in the loop of $X_0$).

We observe that the walk
$(c_n,h_n)$ is just a sum of i.i.d. random variables which are furthermore centered. Indeed, conditioned on the first coordinate being
changed, the expected change is $0$ via
\eqref{eq:3} and the computation by Sheffield \cite{She11} which
boils down to the fact that $\E( |\cR|) =1$ (this is the
quantity $\chi-1$ in \cite{She11}, which is 1 when $q\le 4$) 

Although the change in one coordinate means the other coordinate stays
put, estimating the tail of $\tau$ is actually a one-dimensional problem since the
coordinates are essentially independent. Indeed if instead of changing
at discrete times, each coordinate jumps in continuous time with a
Poisson clock of jump rate $1$, the two coordinates becomes
independent (note that this will not affect the tail exponent by standard 
concentration arguments). Let $\tau^c$ be the return time to $0$ of the first
coordinate. By this argument, $\P(\tau>k) = \P(\tau^c>k)^2$. 
Now, $|\cR|$ has the same distribution as  $J_{T}$ conditionally given $X_{\varphi(0)} = \aF$, 
by \cref{prop:typical_law} equivalence of \cref{typ3,typ4}. It is a standard 
fact that the return time of 
a heavy tailed walk with exponent $\alpha$ has exponent $1/\alpha$. In our 
slightly weaker context, we prove this fact in \cref{lem:return_heavy}. It 
follows 
that $\P(\tau^c> k) = k^{-1/(2p_0) + o(1)}$ and hence $\P( \tau > k) = k^{- 
1/p_0 + o(1)}$. This completes the proof of the tail asymptotics for the length 
of the
loop. \qed


\paragraph{Proof of area exponent in \cref{T:typ}.}  For the lower bound, let us condition on $X_0 = \aF$ and set $T'= - \varphi(0)$. Then we break 
up $T'$ as $T' = \sum_{n=1}^{\tau} (T^c_n +
T^h_n)$ defined as follows. In every reduced walk exploration step, if
the walk moves in the first coordinate, then $T^c_n$ denotes the number
of triangles explored in this step otherwise $T^c_n = 0$. 
Also $T^h_n$
is defined in a similar way. Hence $T_n^{c} + T_n^h$ counts the number of 
symbols explored in step $n$ of the reduced walk. Observe further that translating \cref{prop:loop+area} (third item) to this context and these notations, we have that $\Area(L)=\sum_{n=1}^\tau T_n^c$ or $\Area(L)=\sum_{n=1}^\tau T_n^h$ depending on which coordinate hits zero first (if $\tau = \tau^c$ then $\Area(L)=\sum_{n=1}^\tau T_n^c$ and vice-versa).

Now notice that $T_n^{c/h}$ 
has a probability bounded away from zero to make a jump of size at least $k$ in every 
$k^{p_0+\ve}$ steps, by \cref{thm:exit}. Hence using the Markov property and a union bound over 
cheese and hamburgers,
\begin{equation}
 \P( \sum_{i=1}^{k^{p_0+2\ve}} T^c_i \le k \text{ or }  
\sum_{i=1}^{k^{p_0+2\ve}} T^h_i\le k ) \leq e^{-ck^{\ve}}.
\end{equation}
Hence
\begin{equation}
 \P(\Area(L) > k) \ge \P(\Area(L) > k, \tau >k^{p_0+2\ve}) \ge \P(\tau 
>k^{p_0+2\ve}) - e^{- ck^\eps} \ge k^{-1 - 2\eps + o(1)}
\end{equation}

Now we focus on the upper bound. Since the coordinates are 
symmetric, it is enough to prove
\begin{equation}
 \P(\sum_{n=1}^{\tau^c} T_n^c > k, \, \tau=\tau^c) \le 
k^{-1+\ve + o(1)}.
\end{equation}
Since $\P(\tau>k^{p_0}) =k^{-1+ o(1) }$, we can further restrict ourselves to the case $\tau \le k^{p_0}$. 
Now roughly the idea is as follows.
When we condition on the event $\{\tau = \tau^c =j\}$ with $j \le k^{p_0}$, there are several ways in which the area can be larger than $k$. 

One way is if the 
the maximal jump size of $(|\cR_i|)_{1\le i \le j}$ 
is itself large, in which case there is a maximal envelope with a large boundary (and therefore a large area).

The second way is if the maximal jump size is small and the area manages to be large because of many medium size envelopes, but we are able to discard it by comparing a sum of heavy-tailed random variables to its maximum.

Therefore, the following third way will be the more common. We will see that the maximal jump size in $|\cR_i|$ is at most 
$j^{\frac1{2p_0}}$ with \emph{exponentially high probability}, even though the 
$\cR_i$ are heavy-tailed. Now, if the area is to be large (greater than $k$) and 
one maximal envelope contains essentially all of the area, then the area of that 
envelope will have to be big compared to its boundary. We handle this deviation 
by using a Markov inequality  with a nearly optimal power and \cref{mainitem} in \cref{thm:exit}.

\medskip We first convert the problem to a one-dimensional problem. To this end let 
$ \xi_n =  c_n- c_{n-1}$, i.e., we look at the 
jumps only restricted to the cheeseburger coordinate.
We now observe that on the 
event $\tau^c =k$, we have $\xi^*:=\sup_{n \le \tau^c}\xi_n \le 
k^{\frac{1}{2p_0}+\delta}$ with probability at least $1-k e^{-ck^{\delta}}$. 
To see this we use the following exponential left tail of sums of 
$\xi_n$ (see \cref{lem:left_tail} for a proof; in words, a big jump is exponentially unlikely on the event $\tau^c=j$ 
because if there is one, the walk has to come down  to $0$ very fast)
\begin{equation}
  \label{eq:39}
  \P(\sum_{n=1}^k \xi_n <-\lambda k^{\frac1{2p_0} +\delta}) \le 
2e^{-c(\delta)\lambda}.
\end{equation}
Using all this, it is enough to show, with $\delta = \eps/4$ say,
\begin{equation}
 \P(\sum_{n=1}^{\tau^c} T_n^c > k, \xi^* \le (\tau^c)^{\frac{1}{2p_0} + 
\delta},\, \tau=\tau^c \le k^{p_0}) 
\le k^{-1 + \eps + o(1)}.
\end{equation}
Let $T_j^* = \max_{1 \le n \le 
j} T_n^c$. Using Markov's inequality, for all $\ve>0,\delta=\eps/4$,
\begin{align}
 \P(\sum_{n=1}^{\tau^c} T_n^c > k, & \,\xi^* \le (\tau^c)^{\frac{1}{2p_0} + 
\delta},\, 
\tau=\tau^c \le k^{p_0})\nonumber\\  & \le 
\frac{C(\ve)}{k^{2p_0-\ve}}\sum_{j=1}^{k^{p_0}}\E\left((\sum_{n=1}^{j} 
T_n^c)^{2p_0-\ve}\mathbbm{1}_{\xi^* \le j^{\frac 1 {2p_0} + \delta}}, 
\mathbbm{1}_{\tau=\tau^c =j}\right)\nonumber\\
& \le 
\frac{1}{k^{2p_0-\ve}}\sum_{j=1}^{k^{p_0}} 
\E\left(\left(\frac{\sum_{n=1}^{j } 
T_n^c}{(T_j^*)^{1+\delta}}\right)^{2p_0-\ve}\left((T_j^*)^{1+\delta}\right)^{
2p_0 -\ve} \mathbbm { 1 } _ { \xi^* \le j^{\frac 1 {2p_0} + \delta}}, 
\mathbbm{1}_{\tau=\tau^c =j}\right)\label{eq:holder_trick}
\end{align}
It is a standard fact that for heavy tailed variables with infinite 
expectation, the sum is of the order of its maximum with exponentially high 
probability. This is stated and proved formally in \cref{lem:sum_max}. Using 
this fact, Holder's inequality and the fact that $(1+ \delta) (2p_0 - 
\eps) < 2p_0 - \eps/2$ we conclude that 
\begin{align}
 \P(\sum_{n=1}^{\tau^c} T_n^c > k, \,\xi^* \le (\tau^c)^{\frac 1 {2p_0} + 
\delta},\, 
\tau=\tau^c \le k^{p_0}) & \le 
\frac{C(\eps)}{k^{2p_0-\ve}}\sum_{j=1}^{k^{p_0}}\E\left((T_j^*)^{2p_0-\eps/4
}
\mathbbm { 1 } _ { \xi^* \le j^{\frac{1}{2p_0} + \delta}}, 
\mathbbm{1}_{\tau=\tau^c =j}\right)\nonumber\\
& \le 
\frac{C(\eps)}{k^{2p_0-\ve}}\sum_{j=1}^{k^{p_0}} \E \left(\sum_{1 
\le n \le j} (T_n^c)^{2p_0-\eps/4      }
\mathbbm { 1 } _ { \xi^* \le j^{\frac 1 {2p_0} + \delta}}, 
\mathbbm{1}_{\tau=\tau^c =j}\right)\nonumber
\end{align}
Now let $\cG$ be the $\sigma$-algebra generated by $(\cR_n)_{n \ge 0}$. 
Notice that $\tau^c, \tau^h$, $\xi^*$ are $\cG$-measurable and that $T_n^c$ is independent of 
$(\cR_i)_{i \neq n}$. Also notice from \cref{mainitem} of \cref{thm:exit} that 
$\E((T_n^c)^{2p_0-\eps/4}| \cG) \le C(\eps)|\cR_n|^{4p_0-\eps/4} 
\mathbbm{1}_{\xi_n>0}$. Thus 
we conclude
\begin{align}
  \P(\sum_{n=1}^{\tau^c} T_n^c > k, \,\xi^* \le (\tau^c)^{\frac 1 {2p_0} + 
\delta},\, 
\tau=\tau^c \le k^{p_0})  
& \le 
\frac{C(\eps )}{k^{2p_0-\ve}} \sum_{j=1}^{k^{p_0}}\E \left(\sum_{1 
\le n \le j} |\cR_n|^{4p_0- \eps/4}\mathbbm{1}_{\xi_n>0},
\mathbbm { 1 }_{ \xi^* \le j^{\frac 1 {2p_0} + \delta}}, 
\mathbbm{1}_{\tau=\tau^c =j}\right)\nonumber
\end{align}
Again using Holder and \cref{lem:sum_max} similar to \eqref{eq:holder_trick}, 
we can replace $\sum_{1 
\le n \le j} |\cR_n|^{4p_0-\eps/4}\mathbbm{1}_{\xi_n>0}$ by $(\xi^*)^{4p_0 - 
\eps/8}$ in 
the above expression and obtain that the right hand side above is at most (moving to continuous time to get independence of $\tau^c $ and $\tau^h$ as in the earlier proof of the length exponent),
\begin{align*}
 \frac{C(\eps)}{k^{2p_0-\ve}} \sum_{j=1}^{k^{p_0}}\E 
\left((\xi^*)^{4p_0-\eps/8} \mathbbm { 1 }_{ \xi^* \le j^{\frac 1 {2p_0} + \delta}}, \mathbbm{1}_{\tau=\tau^c =j}\right) 
& \le  \frac{C(\eps)}{k^{2p_0-\ve}} 
\sum_{j=1}^{k^{p_0}} j^{2+\eps}\P(\tau^c = j)\P(\tau^h 
>j) 
\\
& \le \frac{C(\eps)}{k^{2p_0-\ve}} 
\sum_{j=1}^{k^{p_0}}j^{1+2\ve-\frac1{p_0}}
\\
& \le \frac{C(\eps)}{k^{2p_0-\ve}} (k^{p_0})^{2-\frac{1}{p_0}+2\ve}  = k^{-1 + 3 
\eps + o(1)} 
\end{align*}
as desired. \qed
\subsection{Connection with random walk in cone}\label{sec:connection}
Given the sequence $\{X_{i}\}_{i \in
  \mathbb Z}$ and $\cS_0$, the burger stack at time $0$, we can 
construct the sequence 
$\{\hat X_i\}_{i \in \mathbb
  Z}$ where we convert every $\aF$ symbol in $\{X_{i}\}_{i \in
  \mathbb Z}$ into the corresponding $\aC$ symbol or $\aH$
symbol. 

\def\UU{U}
Define $(\UU^x_n)_{n \ge 1}$ to be the 
algebraic cheeseburger count as follows:
\begin{equation}\label{D}
\UU^x_i - \UU^x_{i-1} = 
\begin{cases}
+1&\text{ if }\hat X_i = \ac,\\
-1& \text{ if }\hat X_i = \aC,\\
0 & \text{ otherwise.}
 \end{cases}
 \end{equation}
Similarly define the hamburger count $\UU^y_i$ by letting its
increment $\UU^y_i - \UU^y_{i-1}$ be $\pm1$ depending whether $\hat
X_i = \ah, \aH$ or 0 otherwise.

Recall our notation $p$ defined in \cref{eq:3} so that $p /2= \P( \aF)$. The main result of Sheffield \cite{She11}, which we rephrase for ease of reference later on, is as follows.

\begin{thm}[Sheffield \cite{She11}]\label{thm:she}
Conditioned on any realisation of $\cS_0$, we have the following convergence 
uniformly in every compact interval
\[
\left(\frac{\UU^x_{\lfloor nt \rfloor}}{\sqrt{ n}}, 
\frac{\UU^y_{\lfloor nt
      \rfloor}}{\sqrt{n}}\right)_{t \ge 0} 
            \xrightarrow[]{n \to
  \infty} (L_t, R_t)_{t \ge 0}
\]
where $(L_t, R_t)_{t\ge 0}$ evolves as a two-dimensional correlated Brownian motion with $\var (L_1) = \var (R_1) = (1-p)/2 = \sigma^2$ and $\cov(L_1, R_1) = p/2$. 
\end{thm}

\begin{remark}
Up to a scaling, this Brownian motion $(L_t, R_t)_{t \ge 0}$ is exactly the same which arises in the main result of \cite{LQGmating} (Theorem 9.1).   This is not surprising: indeed, the hamburger and cheeseburger count give precisely the relative length of the boundary on the left and right of the space-filling exploration of the map.  
\end{remark}

In order to work with uncorrelated Brownian motions, we introduce the following linear transformation $\Lambda$:
$$
\Lambda =(1/\sigma)\left( 
\begin{array}{cc}
1 & \cos(\theta_0)\\
0  & \sin (\theta_0)
\end{array}
\right)
$$ 
where $\theta_0= \pi /(2p_0) = 4\pi/\kappa' = 2 \arctan( \sqrt{1/(1-
  2p)}$ and $\sigma^2 = (1-p)/2$ as in the above theorem. A direct but
tedious computation shows that $\Lambda (L_t, R_t)$ is indeed a
standard planar Brownian motion. (The computation is easier to do by
reverting to the original formulation of \cref{thm:she} in
\cite{She11}, where it is shown that $U^x + U^y$ and $(U^x -
U^y)/\sqrt{1-2p}$ form a standard Brownian motion; however this
presentation is easier to understand for what follows).
\begin{figure}
\begin{center}
\includegraphics[scale=.9]{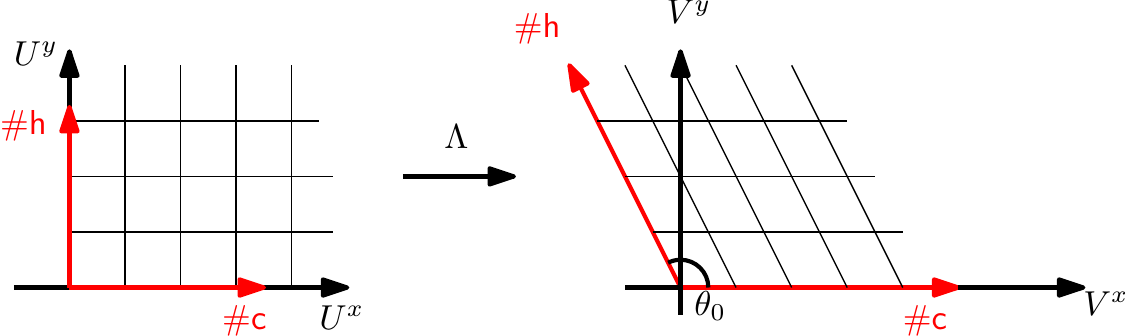}
\caption{The coordinate transformation. Note that in the new coordinates, leaving the cone $\cC(\theta_0)$ (in red in the picture) at some time $n$ corresponds to having eaten all burgers of a given type between times 0 and $n$.}
\label{F:coord}
\end{center}
\end{figure}

\medskip 
We now perform the change of coordinates in the discrete, and thus define
 for $n \ge 0$, $\bd_n = (V^x_n, V^y_n) = \Lambda (\UU^x_n,\UU^y_n)$
 (see \cref{F:coord}). 
 We define $\bv_0 = \Lambda (0,1)$ (note that 
argument of $\bv_0$ is the same as that of the cone). Let $\mathcal C(\theta) 
:= \{(r,\eta): r\ge 0,
 \eta \in  [0,\theta]\}$ denote the $2$-dimensional closed cone of angle
 $\theta$ and let $\cC_n(\theta)$ be the translate of the cone
 $\cC(\theta)$ by the vector $-n\bv_0$. Now 
define $T^*_{\theta_0,n}(\bd):= \min \{k \ge 1, \bd_k \notin \mathcal
C_n(\theta_0)\}$. Let 
$E^*_n$ denote the event that 
$$
E^*_n = (X_0 =
\ac) \cap (\bd_{T^*_{\theta_0,n} - 1} =- n \bv_0 )\cap (X_{T^*_{\theta_0,n}} = \aF)
$$
In words, the walk leaves the cone $\cC_n(\theta_0)$ through its tip, and the 
symbol at this time is an $\aF$. Recall the event $E = \{X_{\ph(0)} = \aF\}$.

\begin{lemma}\label{prop:same_event}
The events $\{X_0 = \ac\} \cap \{T>m^2,|J_T| = n \} \cap  E$ and $  E^*_n \cap \{T_{\theta_0,n}^* > m^2\}$  are identical.
\end{lemma}
\begin{proof}
Consider $(U^x, U^y)$ for a moment and suppose $X_0 = \ac$. Observe that
$\ph(0)$ corresponds to the first time that $U^x = -1$. Moreover, the set 
of times $t \le \ph(0)$ such that $U_t^x = 0$ corresponds to the times at which 
that initial cheeseburger is the top cheeseburger on the stack; and the size of 
the infimum,  $|\inf_{s\le t} U^y_s|$ gives us the number of hamburger orders 
which have its match at a negative time; or in other words, the number of 
hamburger orders $\aH$ in the reduced word at time $t$.

Now the event $E$ occurs if and only if the burger at $X_0$ gets to the top of 
the stack and this is immediately followed by an $\aF$. Hence the event $E \cap 
(|J_T| = n)$ will occur if and only if $ U_t^y = \inf_{s\le t} U^y_s = - n$ and 
$U^x_t = \inf_{s\le t} U^x_s =0$ for some $t$, and we have an $\aF$ 
immediately after. In other words, the walk $(U^x, U^y)$ leaves the quadrant $\{ x \ge 0, y \ge - n \}$ for the first time at time $t$, and does so through the tip. Equivalently, applying the linear map $\Lambda$, 
 $\bd$ leaves the cone $\cC_{n, 
\theta_0}$ for the first time at time $t$, and does so through the tip. 
\end{proof}
\section{Random walk estimates}
\label{S:rw_estimate}

We call $\Lambda (\Z^2)$ the \textbf{lattice points}, which are the points that 
$\bd$ can visit.
Let $s$ be an infinite burger stack and let $x$ be a lattice point.
From now on we denote by $\P^{x,s}$ the law of the walk $\bd$
started from $x$ conditioned on $\cS_0 = s$. In this
section, we prove the following lemma. Recall $p_0 = \pi/2\theta_0$
from \cref{symbols}.

\begin{prop}\label{lem:wrap_up_lemma}
For all $\ve >0$ there exist
positive constants $c= c(\ve),C=C(\ve)$ such that for all $n \ge 1$, all $m \ge n (\log n )^3$, and
any infinite burger stack $s$,
\begin{equation}\label{wrap1}
\frac{c n^{2p_0}}{n^{1+\ve}m^{4p_0+\ve}} \le \P^{0,s}(E^*_n; T^*_{\theta_0,n}>m^2) \le 
\frac{C n^{2p_0}}{n^{1-\ve}m^{4p_0-\ve}}
\end{equation}
Furthermore,
\begin{equation}\label{wrap2}
\frac{c}{n^{2p_0+1+\ve}} \le \P^{0,s}(E^*_n) \le 
\frac{C}{n^{2p_0+1-\ve}}.
\end{equation}
\end{prop}
Using \cref{prop:same_event} and the symmetry betweeen cheese and hamburgers, 
the above lemma completes the proof of
the first item of \cref{thm:exit}.
\subsection{Sketch of argument in Brownian case.} \label{sec:sketch}
To ease the explanations we will first explain heuristically how the exponent can be computed, discussing only the analogous question for a Brownian motion. 
To start with, consider the following simpler question. Let $B$ be a
standard
two-dimensional Brownian motion started at a point with polar
coordinate $(1,\theta_0/2)$ and let $S$ be the first time that $B$
leaves $\cC(\theta_0)$. For this we have:
\begin{equation}\label{coneBM}
\P(S > t )= t^{- p_0+o(1)}
\end{equation}
as $t \to \infty$. To see why this is the case, consider the conformal map $z \mapsto z^{\pi/\theta_0}.$ This sends the cone $\cC(\theta_0)$ to the upper-half plane. In the upper-half plane, the function $z \mapsto  \Im (z)$ is harmonic with zero boundary condition. We deduce that, in the cone,
$$
z\mapsto g(z) : = r^{\pi/\theta_0} \sin \left(\frac{\pi
    \theta}{\theta_0}\right) ; \quad  z\in \cC(\theta_0),
$$
is harmonic.

Now in the cone $\cC_n(\theta_0)$, if Brownian motion survives for a time
$m^2 \gg n^2$, then it is plausible that it reaches distance at least
$m$ from the tip of $\cC_n(\theta_0)$. We are interested in the event
that the Brownian motion reaches distance $m$ from the tip of
$\cC_n(\theta_0)$ before reaching near the tip of $\cC_n(\theta_0)$
while staying inside the cone.

We now decompose this event into
three steps. In the first step, the Brownian motion
must first reach a distance $n/2$ from the origin. This is like
surviving in the upper half plane which by the heuristics above has
probability roughly $n^{-1}$. In the second step, the walk reaches distance
$m$ with probability roughly $(n/m)^{2p_0}$. This can be deduced by
using the harmonic function above which grows like $r^{2p_0}$.

Finally for the third step, the Brownian motion must go back to the tip.
Suppose now, that we are interested in the event $\cE$ that the Brownian motion 
leaves the cone $\cC(\theta_0)$ through the ball of radius 1, that is, $\cE = \{ 
|B_S| \le 1\}$. To compute the tail of $S$ on this event, we can use the function 
$$
z\mapsto g(z) : = r^{-\pi/\theta_0} \sin \left(\frac{\pi \theta}{\theta_0}\right) ;  z\in \cC(\theta_0),
$$
which is harmonic for the same reason as above (note that the sign of
the exponent in the power of $r$ is now opposite of what it was
before (we flipped the images of $0$ and $\infty$ in the choice of the 
conformal map). Using this we can conclude that
coming back to the ball of radius $1$ from distance $m$ costs
$m^{-2p_0}$. 

Thus combining all the three steps, we obtain $n^{-1} \cdot (n/m)^{2p_0}\cdot
m^{-2p_0}$ which is roughly what is claimed in
\cref{lem:wrap_up_lemma}.

\subsection{Cone estimates for random walk}

We want to replace the Brownian motion in the above sketch by a random walk. 
The difficulty here is that the functions $r^{\pm\pi/\theta_0} \sin(\frac{\pi 
\theta}{\theta_0})$ are not exactly harmonic for the random walk. The main idea 
to overcome this is to
approximate the Brownian motion by large blocks of the walk $\bd$, of
an appropriate \emph{macroscopic} length (see \cref{def:Y}) for which the 
central limit theorem will apply. To deal with the small error in this 
approximation, we have to give ourself some room by perturbing the above 
functions so that they become strictly sub-harmonic or super-harmonic and this 
explains why we lose an $\ve$ in the exponent (see \cref{lem:sub_super_mart}). 
Once we know how to get sub-martingale and super-martingale for the walk, the 
rest follows quite easily.
This is similar to a strategy
originally devised by McConnell \cite{mcc}, with some small but
crucial differences.


Now we begin the proof of \cref{lem:wrap_up_lemma}. Recall that $\Lambda(\Z^2)$ 
is
the set of lattice points where $\bd$ can visit. 
For $x>0$ and a process $(Z_k)_{k \ge 0}$ define
$$ 
 T_{\theta_1,\theta_2}^* =T_{\theta_1,\theta_2}^*(Z):= 
\min \{k \ge 1, Z_k \notin \mathcal
C(\theta_1,\theta_2)\}.
$$
We sometimes denote $T^*_\theta$ for $T^*_{0,\theta}$ when there is no source of 
confusion.
Also recall the notation $T^*_{\theta_0,n}$ from
\cref{prop:same_event}. 

%
%

\begin{defn}\label{def:Y}
 For $\ve > 0$ we will 
define the following time-changed
walk $\{Y_i(\ve)\}_{i \ge 0}$ and stopping times $\{\tau_i (\ve)\}_{i \ge 
0}$ 
as follows. Start with $Y_0 (\ve)= \bd_0$ and $\tau_0(\ve) = 0$. Given 
$\bd_{\tau_k(\ve)} = Y_k(\eps) $, we 
inductively
define $\tau_{k+1}(\ve) = \min\{t > \tau_k(\ve): | {\bd}_t
- { \bd}_{\tau_k(\ve)}| > \ve |Y_k|\}$ and $Y_{k+1}(\ve) =
\bd_{\tau_{k+1}}$.
\end{defn}
The next proposition shows that the
Brownian motion estimates in \cref{sec:sketch} pass through to the
discrete walk estimates with $\ve$ error using little more than the invariance
principle. This proposition is the key step for transferring results from Brownian motion
to the discrete walk. To help alleviate notations, and since $\ve$ is
fixed throughout this proposition we will simply write $\tau_k$, and $
Y_k$ for $\tau_k(\eps)$ and $ Y_k(\eps)$. Let $\cC(\theta_1,\theta_2) = \{z :
\theta_1<\arg(z) < \theta_2\} $.
\begin{prop}\label{lem:sub_super_mart}
Fix $\pi/2 \le \theta < \pi$ and an infinite burger stack
$s$. Let $f:\R^2 {\setminus \{0\}} \to \R$ be a continuous function such 
that 
\begin{enumerate}
\item $\Delta f (x)>0$ (resp. $\Delta f (x)<0$) for all $x \in
  \mathcal C(\theta)$.
\item $f$ is \textbf{homogeneous} in the sense that $f(\lambda x) = \lambda^d
x$ for some $d \in \R$ and all $\lambda>0$.
\end{enumerate}
There exists $\eps_0$ such that if $\eps \in (0,\eps_0)$, the following holds.  
There exists a constant $r_0(\ve)>0$ such that
for all lattice points $v $ in $\cC(\ve,\theta-\ve)$ with $|v | > r_0$, $f(Y_{k 
})$ is a $\P^{v,s}$ 
submartingale
(resp. supermartingale) with respect to the filtration $\mathcal F_k=
\sigma (X_i:1 \le i
\le \tau_k)$ until the walk $Y$ exits $\cC(\ve,\theta-\ve) \cap
\{z:|z|>r_0\}$.
\end{prop} 

\begin{proof}
  First observe that it is enough to show that there
  exists an $r_0$ such that for any lattice point $v$ in 
$\cC(2\ve,\theta-2\ve)$ with $|v| > r_0$ and
  any infinite burger stack $s$, 
  \begin{equation}
    \label{eq:5}
    \E^{v,s}\big(f(Y_1)\big) - f(v) >0
  \end{equation}
  where $\E^{v,s}$ is the expectation with respect to the measure
  $\P^{v,s}$. This is because conditioned on $\mathcal F_k$, the sequence
  $\{\bd_{\tau_{k}+i}\}_{i \ge 1}$ has the law  $\P^{\bd_{\tau_k}, \cS_{\tau_k}}$. Also
  notice that we stop when the walk $\bd$ reaches a distance less than
  $r_0$ or leaves the cone $\cC(2\ve,\theta - 2\ve)$ and hence we only need
  to prove \eqref{eq:5} when $v$ is in the claimed range.

The claim \eqref{eq:5} is a consequence of the invariance principle and a result of Sheffield
  (\cref{lem:F} stated below). The important issue is to verify that we can 
pick an $r_0$
  uniformly over the burger stack $s$ and the initial position of the walk
$Y_0= v$. 
 To
  establish this, we set up a few notations. For any continuous curve
  $\gamma$
  let $\tau(\gamma,\ve)$ denote the exit time of $\gamma$ from the ball $\{z: |z-x| < \eps |x|\}$ where $x = \gamma(0)$. 
  Let 
 \begin{equation} \label{eq:Z}
 Z( \gamma) = Z(\gamma,\ve) = 
 \gamma (\tau) 
 \text{ where } \tau = \tau ( \gamma, \ve)
  \end{equation} 
  denote the position of the curve 
  at time $\tau(\gamma,\ve)$.  Let $D =  \overline{\cC(2\ve,\theta - 2\eps) }
\cap  \{ |x| = 1\}$ where $\overline{\cC(2\ve,\theta - 2\eps) }$ is the closure 
of $\cC(2\ve,\theta - 2\eps) $
  (in particular, note that $D$ is compact). With these notations, observe that 
$\E^{v,s} [ f(Y_1)] = \E^{v,s} [ f (Z( \bd))]$, so our goal \eqref{eq:5} becomes
  \begin{equation} \label{eq:5'}
  \E^{v,s}\big[ f (Z ( \bd)) \big] > f(v).
  \end{equation}

  Let $B$ denote a standard Brownian motion. A
  preliminary observation is that by our assumptions on $f$, $\E^x(f(Z(B))) - f(x)  >0$ for all $x \in
  D$. Moreover, since the left hand side is obviously continuous in $x$, 
  we deduce 
that there is a constant $\delta>0$ such that
  \begin{equation}
    \label{eq:4}
    \E^x\big[f(Z(B))\big] - f(x) >\delta, \quad \quad \forall x \in D.
  \end{equation}

  
Now we approximate $B$ by the discrete walk.
Let $A(\ve):=\{x:1-\ve \le |x| \le 1+\ve, 0 \le \arg(x) \le
  \theta\}$  and let  $M = \max \{|f(x)|:x\in A(\ve)\}$. First choose $K>0$ such that
\begin{equation}
 \P^x(\tau(B,2\ve) > K) <\frac \delta {16M} \quad \quad \forall x \in
  D.\label{eq:K}
\end{equation}

Fix some arbitrary infinite
burger stack $s$ for now. Choose $r_0(s)>0$ such that for all 
 $v \in \cC(\eps, \theta - \ve) \cap \Lambda(\Z^2)$ with $|v | > r_0$, the following holds. 
 
 
From the invariance principle (\cref{thm:she}), we know that as $v \to \infty$, 
the distribution of $ \bd ( \cdot |v|^2) /|v|$ under $\P^{v,s}$ is close 
that of Brownian motion $B$ started from $v/|v|$, uniformly over $v$ by translation invariance. 
%
Hence for $v$ sufficiently large (i.e., there is 
$r_0 = r_0(s)$ such that if $|v| > r_0(s)$), by uniform continuity of $f$ in the annulus $A(1/2)$,
\begin{equation}
  \label{eq:6}
  \big|\E^{v,s}\big[ f(Z ( \tfrac{ \bd}{|v|} ))\big] - \E^{v/|v|}\big[f(Z(B))\big] \big| <\delta/8.
\end{equation}
For the same reason, if $|v| > r_0(s)$ then for $K$ as in \eqref{eq:K},
\begin{equation}
  \label{eq:7}
  \P^{v,s}\big[\tau( \bd,2\ve) > K|v|^2\big] ) <\frac \delta {12M}.
\end{equation}
Now we will show that the conclusion \cref{eq:6} holds for $|v| >
r_0(s)$ even if $s$ is changed into another infinite burger stack $s'$. For this the main tool is the following estimate due to Sheffield (which was already at the heart of \cite{She11}).

 \begin{lemma}[Lemma 3.7 in \cite{She11}]\label{lem:F}
Let $F_n$ denote the number of $\aF$ symbols in $\overline{X_0X_1 \ldots X_n}$. 
Then for all $\ve>0$,
\[
 \P\left(\frac{F_n}{\sqrt{n}} >\ve\right) \to 0.
\]
as $n \to \infty$.
\end{lemma}

Fix $\eta \in (0,\ve/2)$ be such that for $x,y$ in the annulus 
$A(1/2)$ 
such that $|x-y|<2\eta$ we have $|f(x) - f(y)| < \delta/4$. Let $h< \eta$ be such 
that $ \P^x[\sup_{\tau(B,  \epsilon-h) \leq t_1, t_2 \leq \tau(B,  
\epsilon+h)} | B_{t_1} - B_{t_2} | > \eta/2] < \delta/(20 M)$ for any $x\in \R^2$. 
Reasoning as in \cref{eq:6} and \cref{eq:7}, we also have that 
\begin{equation} \label{eq:H}
\P^{v,s}\left[  \cH  \right] \le \frac{\delta}{12M}, \text{ where } \cH
=\left\{ \sup_{t_1, t_2 \in [ \tau( \bd,  \eps
    -h), \tau (  \bd ,  \eps+h)] } | \bd_{t_1} - \bd_{t_2}| >  \eta |\bd_0|  \right\}
\end{equation}
(Note for later use that since $\eta$ depends only on $\delta$, $h$ depends only on $\delta$ which depends only on $f$).

Using \cref{lem:F}, we can assume that the choice of $r_0$ (depending only on 
$\delta,M$) is such that for all $r > r_0$, the number $F_r$ of $\aF$ symbols in
$\overline{X_1 \ldots X_r}$ satisfies 
\begin{equation}
  \label{eq:10}
  \P(F_r > h \sqrt{r/K}) \le
\frac{\delta}{12M}.
\end{equation}
 Let $v$ be a fixed 
lattice point in $\cC(\ve,\theta_0-\ve)$ with $|v|>r_0$ with this choice of 
$r_0$. Let $ \bd_t(s) $ denote the walk $\bd$ under $\P^{v,s}$. Observe that if 
$s'$ is another arbitrary burger stack, then 
$$
\sup_{0 \le t \le r}
| \bd_t(s) - \bd_t (s')| \le F_r .
$$ 
Define 
the bad event $\cB$ to be 
$$\cB = \{F_{K|v|^2} > h |v|\} \cup \{\tau( \bd(s),2\ve) >
 K |v|^2\} \cup \cH 
 $$  
On $\mathcal B^c$, the maximal
distance between the paths $ \bd(s)$ and $ \bd (s')$, up to time $K|v|^2$,
is at most $h|v|$. Since on that event we also have $\tau(\bd(s),2\ve) \le 
K|v|^2$, and since $h<\eta<\ve/2$ we deduce $\tau(\bd(s'),\ve) \le K|v|^2$. 
These properties also imply 
$$\tau (  \bd (s'),  \eps ) \in [\tau (  \bd(s) , \eps - h) , \tau ( \bd(s) ,  \eps + h)].
$$ 
Hence by definition of $\cH$, if $\tau = \tau (\bd (s),  \eps)$ and $\tau' = \tau( \bd (s'), \eps)$, 
\begin{align*}
|\frac{\bd_\tau (s)}{|v|} - \frac{\bd_{\tau'} (s')}{|v|} | & 
\le  
|\frac{\bd_\tau (s)}{|v|} -  \frac{\bd_{\tau'} (s)}{|v|} | + 
|\frac{\bd_{\tau'} (s)}{|v|} - \frac{\bd_{\tau'} (s')}{|v|} |\\
&\le \eta + h\le 2\eta.
\end{align*}
Hence by the choice of $\eta$, still on the good event $\cB^c$,
\begin{equation}
  \label{eq:9}
  \left|  f(Z\left(\tfrac{   \bd(s)}{|v|} \right))  - f(Z\left( \tfrac{\bd (s')}{|v|}  \right))\right|< \frac \delta4. 
\end{equation}
But using \cref{eq:H}, \eqref{eq:10} and \eqref{eq:7}, $\P(\mathcal B) <
\delta/(4M)$. Hence  using \eqref{eq:9},
\begin{equation}
  \label{eq:15}
   \left|  \E^{v,s} \big[f(Z\left(\tfrac{   \bd }{|v|} \right))  \big] - \E^{v,s'} \big[ f(Z\left( \tfrac{ \bd }{|v|}  \right)) \big] \right|
 \le \frac \delta{4} + 2M\frac{\delta}{4M}=3\delta/4
\end{equation}
Using \eqref{eq:6} [the desired inequality for the fixed burger stack $s$] and \eqref{eq:15},
$$
  \left| \E^{v,s'} \big[ f(Z\left( \tfrac{ \bd }{|v|}  \right)) \big]  - \E^{v/|v|}\big[ f(Z\left(B\right)) \big] \right|  \le
  \frac{7\delta}{8} 
$$
Combining with \cref{eq:4} [the inequality for Brownian motion], we deduce that 
$$
 \E^{v,s'} \big[ f(Z\left( \tfrac{ \bd }{|v|}  \right)) \big]  \ge f(\tfrac{v}{|v|}) + \delta/8.
$$
Using homogeneity of $f$,
\begin{equation}
  \label{eq:14}
\E^{v,s'} [ f (Z( \bd))] = |v|^d \E^{v,s'} \big[ f(Z\left( \tfrac{ \bd }{|v|}  \right)) \big]  > |v|^{d} f( \tfrac{v}{|v|}) = f(v).
\end{equation}
This proves our claim \cref{eq:5'} which, as discussed earlier, implies the proposition. 
\end{proof}

We can now begin the proof of \cref{lem:wrap_up_lemma}. We will focus on \cref{wrap1} as the proof of \cref{wrap2} is identical (with only steps 1 and 3 below needed).
We start by recalling the formula for the Laplacian in polar coordinates which we will use repeatedly: if $f(r,\theta):= r^d\varphi(\theta)$ where $d \in \R$,
\begin{equation}
\Delta f(r,\theta) = r^{d-2}(d^2\ph(\theta) + \ph''(\theta))\label{eq:Lap}.
\end{equation}

We use perturbations of the harmonic functions in the cone as sketched in 
\cref{sec:sketch} to construct appropriate supermartingales.

For $x>0$ and a process $(Z_k)_{k \ge 0}$ define
\begin{align}\label{e:T}
T_x^+(Z) &= \min\{k \ge 1, |Z_k| \ge x\}, \quad  T_x^-(Z) = \min \{k \ge 1,
|Z_k| \le x\}, \\ 
T_{x,n}^-(Z) &= \min \{k \ge 1,
|Z_k + n\bv_0| \le x\}.
\end{align}
\paragraph{Step 1} (going out to distance $n/2$). $ \displaystyle{\P^{0,s}\big(  
T^+_{n/2}(\bd) < T^*_{\theta_0,n} (\bd) \big)
\le \frac{C(\ve)}{n^{1-\ve}} }$:

\smallskip
\begin{proof}[Proof of Step 1.] {Recall that this probability is roughly 
the probability to go to distance $n/2$ in some half plane before returning to 
$0$.}   Choose $\ve$ small enough so that $\theta_0+\ve<\pi$.
Consider the cone $\cC(\theta_0 - 
\pi-\ve,\theta_0+\ve)$. Consider the function
\begin{equation}
 g_1^\u(r,\theta) := r^{1-\ve}\sin\left(\frac{\pi}{\pi+2\ve}(\theta - 
\theta_0+\pi+\ve)\right).\nonumber
\end{equation}
We can assume $\ve>0$ is small enough so that $1-\ve< \frac{\pi}{\pi+2\ve}$. It 
is easy to check by \eqref{eq:Lap} that $\Delta g_1^\u <0$ and $g_1^\u >0$ in 
$\cC(\theta_0 - \pi-\ve,\theta_0+\ve)$. By  \cref{lem:sub_super_mart} we can choose $r_0(\ve)$ large enough (depending on $g_1^\u$)  
so that $g_1^\u(Y(\ve))$ is a supermartingale until it 
leaves $\cC(\theta_0 - \pi-\ve,\theta_0+\ve)$. 

Let 
$n':= n(1-\ve)$. Let $$\tau = T^+_{n'/2}(Y(\ve/2)) \wedge T^*_{\theta_0 - 
\pi,\theta_0} (Y(\ve/2)) \wedge T^-_{r_0}(Y(\ve/2))$$ and 
let $v$ be any lattice point in $\cC(\theta_0-\pi,\theta_0)$ such that $r_0<|v| < 3r_0$. Observe that 
$$
T^+_{n/2}(\bd) 
< T^*_{\theta_0 - \pi,\theta_0} (\bd) \wedge 
T^-_{r_0}(\bd)  \text{ 
implies }   T^+_{n'/2}(Y(\ve)) 
< T^*_{\theta_0 - \pi,\theta_0} (Y(\ve/2)) \wedge T^-_{r_0}(Y(\ve/2)) 
$$
Hence since $g_1^\u$ is nonnegative, we obtain by optional stopping theorem, 
\begin{align}
 (3r_0)^{1-\ve}\ge g_1^\u(v) & \ge \E^{v,s}(g_1^\u(Y_\tau(\ve/2))\nonumber\\
& \ge a (n'/2)^{1-\ve}\P^{v,s}(T^+_{n'/2}(Y(\ve/2)) 
< T^*_{\theta_0 - \pi,\theta_0} (Y(\ve/2)) \wedge T^-_{r_0}(Y(\ve/2)))\nonumber\\
& \ge  c n^{1-\eps} \P^{v,s}(T^+_{n/2}(\bd) 
< T^*_{\theta_0 - \pi,\theta_0} (\bd) \wedge 
T^-_{r_0}(\bd) )\label{eq:upper1}
\end{align}
where $a$ is the minimum value of the angular part of $g_1^\u$ on $\cC_{\theta_0- \pi - \eps/2, \theta_0 + \eps/2}$. In particular, $a$ and thus the constant $c$ is bounded below independently of $n$, and depends only on $\eps$ as desired. 
This proves the required bound for the walk starting from any vertex $v$ at a distance between $r_0$ and $3r_0$ which is also stopped if it comes within distance $r_0$ or the origin. 

Now we argue that this additional stopping does not matter. 
Indeed, if the walk reaches distance $n'/2$ it reaches distance more than $r_0$ at some point. Let $N$ be the number of intervals of times the walk is within distance $r_0(\ve)$ before $T^*_{\theta_0,n}$. Since the walk has $c(\ve)$ probability to exit the cone from within distance $r_0$ before reaching distance $2r_0$, we see that $N$ has exponential tail (with constants depending only on $\ve$). 
So we can restrict to the event $N \le \log^2n$: more precisely, by a union bound,
$$
\P^{0,s}( T^+_{n/2}(\bd) < T^*_{\theta_0,n}) \le \P( N \ge (\log n)^2) + (\log n)^2 \sup_v  \P^{v,s}(T^+_{n/2}(\bd) 
< T^*_{\theta_0 - \pi,\theta_0} (\bd) \wedge 
T^-_{r_0}(\bd) )
$$
where the sup is over vertices $v$ at distance between $r_0$ and $3r_0$ from the origin. We deduce from \cref{eq:upper1} which is uniform over $v$ in this range the desired upper bound.
\end{proof}

\paragraph{Step 2} (from distance $n/2$ to $m$). $\displaystyle{ \P^{v,s}(T^+_m(\bd) < T^*_{\theta_0,n}(\bd)) 
\le \frac{C(\ve)n^{2p_0+\ve}}{m^{2p_0 - \ve}}}$   
for any vertex $v$ with 
$n/4<|v+ n \bv_0| < 3n/4$, and $m \ge n (\log n)$.

\begin{proof}[Proof of Step 2.] This is similar to step $1$, with a few differences as follows. First, by translation invariance, it suffices to prove the result with $\cC_{n}(\theta_0)$ replaced by $\cC(\theta_0)$ and $(n/4) \le |v| \le 3n/4$. Consider the function (recall $2p_0 = \pi/\theta_0$)
\begin{equation*}
 g_2^\u(r,\theta) := 
r^{2p_0-5\ve}\sin\left(\frac{\pi}{\theta_0+2\ve}(\theta+ \ve)\right)
\end{equation*}
Clearly $g_2^\u > 0$ and $\Delta g_2^\u<0$ in $\cC(-\ve,\theta_0+\ve)$ by 
\eqref{eq:Lap}. By  
\cref{lem:sub_super_mart} we can choose $r_0$ such that $g_2^\u (Y(\eps/2))$ is a supermartingale until it leaves this cone. Let 
$$
\tau = T^*_{\theta_0}(Y(\ve/2)) \wedge 
T^+_{m(1-\ve)}(Y(\ve/2))\wedge T^-_{r_0}(Y(\ve/2)).
$$  
Let $v$ be any vertex with $n/4<|v| < 3n/4$. By optional stopping theorem, 
writing a similar chain of inequalities as in Step 1 :
\begin{equation}
n^{2p_0 - 5\eps}\ge  |v|^{2p_0-5\ve} \ge g_2^\u(v) \ge  c m^{2p_0-6\ve}\P^{v,s}(T^+_{m}(\bd) < 
T^*_{\theta_0}(\bd) \wedge T^-_{r_0}(\bd)).
\end{equation}
We complete the proof by the same argument as in step 1 (showing that the time spent within $B(0, r_0)$ before $T^*_{\theta_0}$ has exponential tail).\end{proof}

\paragraph{Step 3} (from distance $m$ to the tip of the cone). $\displaystyle{ \P^{v,s}(T^-_{r_0}(\bd) < 
T^*_{\theta_0,n}(\bd)) \le \frac{C(\ve)}{m^{2p_0 - \ve}}} $
 for any vertex 
$v$ with $|v| > m$ and $m > n(\log n)$.

\begin{proof}[Proof of Step 3.]
 For this step, again by translation invariance we replace $\cC_n(\theta_0)$ by 
$\cC(\theta_0)$ and assume that the starting point $v$ is at a distance $m$ 
from the origin since $m \ge n (\log n)^3$. 
Consider the function
\begin{equation*}
 g_3^\u(r,\theta) := 
r^{-2p_0+5\ve}\sin\left(\frac{\pi}{\theta_0+2\ve}(\theta+ \ve)\right)
\end{equation*}
and observe that now the exponent in the power of $r$ is negative. 
Using a similar chain of arguments as in steps 1 and 2 (note that the harmonic 
function used here is bounded so we can use optional stopping), we obtain
\begin{equation}
  m^{-2p_0+5\ve} \ge g_3^\u(v) \ge c \P^{v,s}(T_{r_0}^-(\bd) < 
T^*_{\theta_0}(\bd))
\end{equation}
for some constant $c = c(\eps)$.
 
\end{proof}

To put together these three steps and finish the proof of the upper bound in \cref{wrap1}, we need the following observation.

\begin{lemma}\label{lem:time_bound}
Fix an infinite burger stack $s$. There exist positive constants
$c,c'$ (independent of $s$) such that,
\[
\P^{0,s}\left(T^+_{m (\log m)^2}(\bd) \ge m^2, 
T^+_{\frac{m}{\log m}}(\bd) \le
  m^2  \right) \ge 1-c\exp(-c'\log^2m)
\]
\end{lemma}
\begin{proof}
We are going to drop $\bd$ from $T^+_{m (\log m)^2}(\bd), 
T^+_{m/\log m }(\bd)$ to ease notation.
Using Lemma 3.12 of
  \cite{She11} (which proves that the probability of the walk 
$|\bd_n|$
  being greater than $a\sqrt{n}$ is at most $ce^{-c'a}$), we see that 
$$\P^{0,s}\left(T^+_{m(\log m)^2} \le
  m^2 \right) \le ce^{-c'\log^2m}.$$

On the other hand, from any lattice point, the walk $\bd$ has a
positive probability to reach a vertex $v$ at distance $r_0$ from the origin. By the invariance principle \cref{thm:she}, $$\P^{0,s} ( |\bd_{m^2/ (\log m)^2} | \ge m/\log m) \ge c_0.$$
If the
walk fails to reach distance $m/\log m$, let $v'$ be the position of the walk at time $m^2/(\log m)^2$, and let $s'$ be the burger stack at that time. We now iterate this argument by using the Markov property and 
and
\cref{lem:F}. Let $t_i = i m^2 /(\log m)^2$, $i = 1, \ldots (\log m)^2$. Call a time $t_i$ good if the following two conditions hold:
\begin{itemize}
\item  the number of $\aF$ symbols in the reduced word $\overline{X_{t_i}\ldots X_{t_{i+1} - 1}}$ is less than $m/( \log m)$;  
\item $|\bd'_{t_{i+1} -t_i}| \ge 2m/( \log m)$, where $\bd'$ is the walk corresponding to the symbol sequence 
$\{X_{t_i+j}\}_{j\ge 1}$ with the fixed initial stack $ s$.
\end{itemize}
Note that for each $i\ge 1$, conditionally on $(X_0, \ldots X_{t_i - 1})$, $t_i$ is good with probability at least $c_0/2$ by \cref{lem:F}. Furthermore if one of the $t_i$ is good then $\bd$ must have reached distance $m/\log m$.
This completes the lemma.
\end{proof}

\begin{proof}[Proof of upper bound in \cref{wrap1}]
We combine all three steps together using \cref{lem:time_bound}. 
Observe that
\begin{align}
\P^{0,s}( T^*_{\theta_0,n}> m^2 ; E^*_n) &\le  \P^{0,s}( T^*_{\theta_0,n}> m^2 ; E^*_n ; T^+_{m/\log m} \le m^2) +   \P^{0,s} (T^+_{m/\log m} > m^2)
\\
& \le \P^{0,s}(T^+_{n/2} < T^*_{\theta_0, n})\sup_{v,s'} \P^{v,s'} (T^+_{m/\log 
m} < T^*_{\theta_0, n}) \sup_{v',s''}\P^{v',s''}(T^-_{r_0} < T^*_{\theta_0, 
n})\nonumber\\
& \quad \quad  + \P^{0,s} (T^+_{m/\log m} > m^2)\label{e:negli}.
\end{align}
where the sups are respectively over burger stacks $s'$ and vertices $v$ at 
distance $n/2 + O(1)$; and  burger stacks $s''$ and vertices $v'$ at distance 
$m/\log m + O(1)$. Note that the final term on the right hand side of 
\eqref{e:negli} is negligible compared to the first term via 
\cref{lem:time_bound} which completes the proof of upper bound.
\end{proof}

\def \l {\text{lower}}
We now begin the proof of the \textbf{lower bound} of \cref{lem:wrap_up_lemma}.
 The strategy is the same as that in the proof of upper bound except now we 
need to perturb the harmonic functions in \cref{sec:sketch} so as to obtain 
submartingales which takes negative values in a neighbourhood of the 
boundary of the cone.

We will need to lower bound the probability that the walk exits the ball of radius 
$n/2$ from $0$ before exiting a cone which has angle slightly less than $\pi$. 
Consider the following harmonic function
\begin{equation}
 g_1^\l (r,\theta) := r^{1+\ve}\sin\left( \frac{\pi}{\pi-2\ve}(\theta - \theta_0+\pi-\ve) 
\right)
\end{equation}
Since $1+\ve>\frac{\pi}{\pi-2\ve}$ for all small enough $\ve$, we see that 
$\Delta g_1^\l >0$ in $\cC(\theta_0-\pi+\ve,\theta_0-\ve)$. Note that 
$g_1^\l<0$ just outside the boundary of the cone which is desirable, but 
$\Delta g_1^\l<0$ just outside the boundary of this cone, that is $\Delta 
g_1^\l<0$ in $\cC(\theta_0-\pi,\theta_0-\pi+\ve)$ and 
$\cC(\theta_0-\ve,\theta_0)$ which is not desirable. So we wish to modify 
$g_1^\l$ slightly to make $\Delta g_1^\l>0$ not only in the cone $\cC(\theta_0-\pi+\ve,\theta_0-\ve)$ but also in some slightly bigger cone. We moreover wish to do so while keeping the values of the function on the boundaries of this bigger cone negative.

We show 
how to do this modification in a neighbourhood of $\{z:\arg(z) = 
\theta_0-\ve\}$ while in the other boundary the modification follows the same 
trick which we will skip.

For notational convenience let $\ph(\theta) = 
\sin\left(\frac{\pi}{\pi-2\ve}(\theta - \theta_0+\pi-\ve)\right)$ to be the 
angular part of $g_1^\l$. The planned modification is achieved in the following 
tedious but elementary single variable calculus 
problem.

\begin{lemma}\label{claim:negative}
For all $\ve > 0$ small enough, there exist $\delta =\delta(\ve)\in (0,\ve)$ 
small 
enough, and constants $a(\ve),b(\ve),c(\ve),d(\ve)$ such that, if $\tilde \theta = 
\theta_0 - \eps - \delta$, and if
$$
\tilde \ph(\theta) = \begin{cases}
\ph(\theta) & \text{ if } \theta\le \tilde \theta \\
s(\theta) &\text{ if }\theta\in [\tilde \theta, \theta_0]
\end{cases}
$$
where 
 $$
 s(\theta) := a +b (\theta -
\tilde \theta) +c (\theta - \tilde \theta)^2 + d (\theta -
\tilde \theta)^3;\quad \quad \theta \in [\tilde \theta , \theta_0]
$$
then $\tilde g(r, \theta) = r^{1+\eps} \tilde \ph(\theta)$ defines a $C^2$ function in the cone $\cC(\theta_0 -\pi + \eps, \theta_0)$ and moreover $\tilde \ph (\theta_0) <0$ and $\Delta  \tilde g > 0 $ in this cone.
\end{lemma}
\begin{proof}
Observe that the function above is trivially $C^2$ if we choose $a = \varphi(\tilde \theta)$, $b = 
\varphi'(\tilde \theta)$ and $c = \varphi''(\tilde \theta)$ for any choice of 
$\delta \in (0,\ve)$. We now assume this in the following. 
Observe also that by construction $\ph(\theta_0-\ve) = 0$ so as $\delta \to 0$, 
$a \sim p\delta$ where $p = \pi/(\pi - 2\eps)$, $b = -p + o(1)$ and $c \sim -p^3 \delta$. Furthermore all 
the smaller order terms can be bounded independently of $\ve$. In particular if 
we take $\delta = (p/3) \ve$, we have for $\ve$ small enough, $a < 
\frac{p}{2}\ve$, $b < -2p/3$ and $- 8 \ve< c < 0$ (we have $p < 2$).
%
 Now fix $d = 1/(p\ve)$. 
 This choice ensures that
\begin{align*}
  s(\theta_0) = a + b(\ve + \delta) + c(\ve + \delta)^2 + d(\ve + \delta)^3 
  \le \frac{p}{2} \ve -\frac{2}{3}p\ve + 8 \ve^2/p
\end{align*}
which is negative for $\ve$ small enough.

Now let us control the Laplacian, recalling its expression $\Delta \tilde g 
= r^{p'-2}(p'^2 \tilde \ph(\theta)+ \tilde \ph''(\theta))$ with $p' = 1+\eps$. By a Taylor expansion with explicit 
remainder we have for all $\eps>0$, $\theta\in [\tilde \theta, \theta_0]$,
\begin{align*}
  \Big| \varphi(\theta )-s(\theta) - (\theta -\tilde \theta)^3(\varphi'''(\tilde \theta) 
-6d)\Big|
&
 \leq (p^4/4!)(\theta-\tilde \theta)^4
\end{align*}
Also $\phi'''(\tilde \theta) \to p^3 < 6d$ as $\ve\to 0$. 
Therefore if $\ve$ is small 
enough, then for all $\theta \in [\tilde \theta, \theta_0- \eps]$ we have $s(\theta) > \ph(\theta)$. 
Likewise,
\begin{align*}
\Big|\varphi''(\theta) - s''(\theta) -(\theta-\tilde \theta) (\varphi'''(\tilde \theta) 
-6d )\Big|
&\le (p^4/2!)(\theta -\tilde \theta)^2
\end{align*}
and thus $s''(\theta) > \ph''(\theta)$ on $[\tilde \theta, \theta_0 - \eps]$. Consequently,
for $\theta\in[\tilde \theta, \theta_0 - \ve)$, recalling $p' = 1+\eps$, we have 
$$p'^2s 
(\theta)+s''(\theta)
> p'^2 \varphi(\theta) + \varphi''(\theta)>0.$$
Furthermore for $\theta 
\in[\theta_0-\ve, \theta_0]$, we have $(p/3) \ve \le \theta-\tilde \theta \leq 2\ve$ 
and
\begin{align*}
p'^2s(\theta) +s''(\theta) & = \big(p'^2
(a+ b (\theta - \tilde \theta) + c(\theta - \tilde \theta)^2 + d( \theta - \tilde \theta)^3\big) + 2c + 6d
(\theta - \tilde \theta) \\
& >-p'^2 (|b|2\ve + |c|(2\ve)^2) - 2|c| +6\frac{1}{p\ve}\frac{p}{3}\ve >0
 \end{align*}
which concludes the proof.
\end{proof}

With this lemma in hand we can start the first step of the proof of the lower bound in \cref{lem:wrap_up_lemma}.

\paragraph{Step 1.} (going out to distance $n/2$). There exists 
$0<\eta_1<\ve$,
\begin{equation}
\P^{0,s}(T^{+}_{n/2}(\bd) <T^*_{\theta_0-\pi+\eta_1,\theta_0-\eta_1}(\bd)) \ge 
\frac{c(\ve)}{n^{1+\ve}}.
\end{equation}

\begin{proof}[Proof of Step 1.]We choose $\ve$ small enough and $\delta(\ve)>0$ 
as in 
\cref{claim:negative} and replace $\ph(\theta)$ by $s(\theta)$ for $\theta \in 
(\theta_0-\ve-\delta,\theta_0]$ and by an analogous function in 
$(\theta_0-\pi+\delta+\ve,\theta_0-\pi]$ and still call the modified function 
$g_1^\l$ by an abuse of notation. By construction, $\Delta g_1^\l>0$ in the 
interior of $\cC(\theta_0-\pi,\theta_0)$ and 
$g_1^\l(r,\theta_0-\pi)<0,g_1^\l(r,\theta_0)<0$ for any $r>0$. Further, 
$g_1^\l>0$ at any point 
with angle $\theta_0-\pi+\delta+\ve$ or $\theta_0-\delta-\ve$. By continuity, 
$g_1^\l(r,\theta_0-2\delta_1)=0$ for some $\delta_1>0$ and 
$g_1^\l(r,\theta_0 -\pi+2\delta_2) = 0$ for some $\delta_2>0$. Now consider the 
walk $Y(\eta)$ where $\eta = \delta_1\wedge \delta_2$ and let $r_0$ be chosen 
as in \cref{lem:sub_super_mart} for $g_1^\l$, $Y(\eta)$. Clearly, 
there is a probability at least $c(\ve)$ to reach a vertex at distance $3r_0$ from 
$0$ and remaining well within the cone so we can assume we start from such a 
vertex $v$. Let $\tau = 
T^*_{\theta_0-2\delta_1,\theta_0 -\pi+2\delta_2}(Y(\eta)) \wedge 
T^+_{n/2}(Y(\eta)) \wedge T^-_{r_0(1+\eta)}(Y(\eta))$. Clearly
$$
T^+_{n/2}(Y(\eta))<T^-_{r_0(1+\eta)}(Y(\eta)) \wedge 
T^*_{\theta_0-2\delta_1,\theta_0 -\pi+2\delta_2}(Y(\eta)) \text{ implies 
}T^+_{n/2}(\bd)<T^-_{r_0}(\bd) \wedge T^*_{\theta_0-\delta_1,\theta_0 
-\pi+\delta_2}(\bd)
$$  
Thus applying optional stopping and the fact that $g_1^\l(Y(\eta))$ is a 
submartingale until it leaves the cone or comes within distance $r_0$ to the 
tip (note that $g_1(Y(\eta))$ is bounded up to $T^+_{n/2}(Y(\eta)$ so the application of optional stopping is valid):
\begin{align*}
 (2r_0)^{1+\ve} \le g_1^\l(v) & \le 
(3n/4)^{1+\ve}\P^{v,s}(T^+_{n/2}(Y(\eta))<T^-_{r_0(1+\eta)}(Y(\eta)) \wedge 
T^*_{\theta_0-2\delta_1,\theta_0 -\pi+2\delta_2}(Y(\eta)) ) +(r_0)^{1+\ve}\\
& \le (3n/4)^{1+\ve}\P^{v,s}(T^+_{n/2}(\bd)<T^-_{r_0}(\bd) \wedge 
T^*_{\theta_0-\delta_1,\theta_0 
-\pi+\delta_2}(\bd)) +(r_0)^{1+\ve}
\end{align*}
Also notice that if $T^+_{n/2}(\bd)<T^-_{r_0}(\bd) \wedge 
T^*_{\theta_0-\delta_1,\theta_0 
-\pi+\delta_2}(\bd)$ then 
$T^+_{n/2}(\bd)<T^*_{\theta_0-\pi+\eta_1,\theta_0 - \eta_1}(\bd)$ for some 
$\eta_1>0$ thereby completing the proof of this step. \end{proof}

Note that if $T^+_{n/2}(\bd)<T^*_{\theta_0-\pi+\eta_1,\theta_0 - 
\eta_1}(\bd)$ then the walk is in a vertex $v_1$ which is at least distance 
$n/2$ from $-n\bv_0$ and in $\cC_n(\eta_1', \theta_0-\eta_1')$ for some 
$\eta_1'>0$.

\paragraph{Step 2.}(From distance $n/2$ to $m$) 
There exists $0<\eta_2<\eta_1'$ such that 
$$
\P^{v_1,s}(T^+_m(\bd) < 
T^*_{\eta_2,\theta_0-\eta_2,n}(\bd)) \ge \frac{cn^{2p_0-\ve}}{m^{2p_0+\ve}}$$ 
for any $v_1 \in \cC_n(\theta_0-\eta_1',\eta_1')$ with $|v_1|>n/4$ and $m \ge 
n\log n$.

\begin{proof}[Proof of Step 2]
By translation invariance and since $m \gg n$, we can replace 
$\cC_n(\theta_0-\eta_1',\eta_1')$ by $\cC(\theta_0)$ with $v_1$ at a distance at 
least $n/4$ from the origin. Consider the function
\begin{equation}
 g_2^\l (r, \theta):= r^{2p_0+5\eta_1'}\sin\left( 
\frac{\pi}{\theta_0-2\eta_1'}(\theta-\eta_1')\right).
\end{equation}
As in step 1, this has positive Laplacian and takes negative values just 
outside the cone $\cC(\eta_1',\theta_0-\eta_1')$. Using the same argument as in 
\cref{claim:negative} we can modify $g_2^\l$ so that its value is positive 
inside $\cC(\eta_1',\theta_0-\eta_1')$ and negative just outside 
$\cC(\delta_3,\theta_0-\delta_4)$ for some $\delta_3<\eta_1'$ and 
$\delta_4<\eta_1'$. Now following the same strategy as in Step 1 (using the 
submartingale $g_2^\l(Y)$), we find the following chain of inequalities 
\begin{align*}
c(n/4)^{2p_0+5\eta_1'} \le g_2^\l(v_1) \le m^{2p_0+5\eta_1'}\P^{v_1,s}(T^+_m(\bd) 
< T^-_{r_0}(\bd) \wedge T^*_{\delta_3,\theta_0-\delta_4}(\bd)) + 
(r_0)^{2p_0+5\eta_1'} 
\end{align*}
for $n$ large enough and where $r_0$ depends just on $\eps$. We complete by rearranging.\end{proof}

At the end of step 2, suppose we reach a vertex $v_2$ with $|v_2|>m$ and $v_2 
\in \cC(\eta_2,\theta_0-\eta_2)$ where $\eta_2 = \delta_3 \wedge \delta_4$.
Let $T_0(\bd)$ be the first time the walk hits the origin. 

\paragraph{Step 3.} (From 
distance $m$ to the origin). There exist a constant $c$ depending only on $\eps$ such that for $n$ large enough and $m \ge n \log n$, 
$$
\P^{v_2,s}(T_0(\bd)<T^*_{\theta_0,n}(\bd)) \ge 
\frac{c}{m^{2p_0+5\eta_2}}
$$ 
for any $v_2 \in \cC(\eta_2,\theta_0-\eta_2)
$ with 
$m \le |v_2| \le 2m$.
\begin{proof}
[Proof of Step 3.] Again by translation invariance, we can replace 
$\cC_n(\theta_0)$ by $\cC(\theta_0)$. Now we consider the harmonic function 
similar to step 2, but with opposite sign in the exponent.
\begin{equation}
 g_3^\l (r, \theta) := r^{-2p_0 -5 \eta_2 }\sin\left( 
\frac{\pi}{\theta_0-2\eta_2}(\theta-\eta_2)\right)
\end{equation}
Proceeding as before (i.e. modifying $g_3^\l$ using \cref{claim:negative} and 
 choosing a slightly larger cone to work with and choosing $r_0$ 
appropriately using \cref{lem:sub_super_mart}), then using the optional stopping theorem (note that $g_3^\l$ is bounded at distances greater than $r_0$), we obtain
\begin{equation}
 (2m)^{-2p_0-5\eta_2} \le g_3^\l(v_2) \le 
c(r_0)^{-2p_0-5\eta_2}\P^{v_2,s}(T^-_{r_0}(\bd) < T^*_{\theta_0}(\bd))
\end{equation}
This completes the proof because we can reach $0$ from distance $r_0(\ve)$ with 
probability at least $c(\ve)$.
\end{proof}
\begin{proof}[Proof of lower bound of \cref{lem:wrap_up_lemma}]
 Notice that
\begin{equation}
 \P^{0,s}(T^*_{\theta_0,n}(\bd) > m^2,E_n^*) \ge \P^{0,s}(T^*_{\theta_0,n}(\bd) 
> T^+_{m/\log m}(\bd),E_n^*) -\P^{0,s}(T^+_{m/\log m}(\bd) > m^2)
\end{equation}
The first term above has the right lower bound using the three steps executed 
before, while the last term above is negligible compared to the first term 
using \cref{lem:time_bound}. This completes the proof.
\end{proof}

\begin{proof}
 [Proof of \cref{thm:exit}]
 Notice \cref{item4} follows easily from \cref{item1,item2}. 
\cref{mainitem} follows easily by summing the tail estimates of  
\cref{item4}. 
Finally we obtain \cref{item3} from \cref{item1,item2} by summing over $n$. More precisely,
\begin{align}
 \P^s(T>m^2,E) & \le \sum_{n=1}^{m^{1-\ve}}\P^s(T>m^2,|J_T| = n,E) + 
\sum_{n>m^{1-\ve}} \P^s(|J_T| = n,E)\nonumber\\
& \le C \sum_{n=1}^{m^{1-\ve}}\frac{n^{2p_0-1+\ve}}{m^{4p_0-\ve}} + C\sum_{n=m^{1-\eps}}^\infty \frac1{n^{2p_0 +1 - \eps}}\nonumber\\
& \le C\frac{1}{m^{2p_0-2\ve}} + C \frac{1}{m^{2p_0 - 3\eps}}\nonumber
\end{align}
The lower bound follows similarly, by ignoring terms corresponding to $n \ge 
m^{1-\eps}$ and using the corresponding lower bound for the remaining terms. 
This concludes the proof of \cref{thm:exit} and therefore also the proof of 
\cref{T:typ}.
\end{proof}

\appendix
\section{Some heavy tail estimates}\label{A}
In this appendix, we record some lemmas about certain exponents related to 
heavy-tailed random walks. These are standard when the step distribution has 
regular variation, but we need a slight extension without this assumption, and 
while this is probably well known we could not find a reference.

\begin{lemma}
 \label{lem:left_tail}
Let $X_1,X_2\ldots $ be i.i.d.\ with $\P(X_1 \ge n) = n^{-\alpha+o(1)}$ for 
some $\alpha \in (1,2)$, $\E(X_1) = 0$ and $X_1 \ge -1$. Let $S_n = 
\sum_{i=1}^n X_i$. For all $\ve>0$, there exists a $c=c(\eps)>0$ such that 
for any $n \ge 1, \lambda >1$,  
\[
 \P(S_n \le - \lambda n^{\frac1{\alpha-\ve} } )\le 2e^{-c\lambda}
\]
\end{lemma}
\begin{proof}
Fix $\ve>0$.
We are going to consider the truncated variables $X'=X\mathbbm{1}_{X \le 
n^{\frac{1}{\alpha-2\ve}}}$. Since $X$ dominates $X'$, it is enough to prove 
the bound for $S_n'$, the partial sum of $n$ i.i.d.\ variables distributed as 
$X'$. 
Since $\E(X) = 0$, an easy computation yields that $\E(S_n') = 
O(n^{\frac{1-\ve}{\alpha-2\ve}}) = o(\lambda n^{\frac{1}{\alpha-2\ve}})$. 
Similarly, $\var(S_n) = O(n^{\frac{2-\ve}{\alpha - 2\ve}})$. So applying 
Bernstein's inequality, we see that
\begin{align}
 \P(S_n'\le -\lambda n^{\frac{1}{\alpha-2\ve}}) \le \P(S_n'-\E(S_n') \le 
-\frac\lambda 2 n^{\frac{1}{\alpha-2\ve}}) \le 
2\exp\left (-\frac{\frac 1 2 \frac{\lambda^2}{4} n^{\frac{2}{\alpha-2\ve}}}{ 
O(n^{\frac{2-\ve}{\alpha - 2\ve}}) + \frac 1 3 \frac{\lambda}{2}   
n^{\frac{2}{\alpha-2\ve}}} \right) \le 2e^{-c\lambda}
\end{align}
as desired.
\end{proof}

\begin{lemma}
  \label{lem:return_heavy}
Fix $\alpha \in (1,2)$. Let $\{Z_i\}_{i \ge 1}$ be an i.i.d.\ sequence
of integer valued random variables
with $Z_1 \ge -1$ and  $\E(Z_1) = 0$. Suppose
for all $\ve>0$ such that $\alpha -\ve >1$ and $\alpha +\ve<2$, there
exist $c(\ve),C(\ve)>0$  such that for all $n\ge 1$
\[
\frac{c(\ve)}{n^{\alpha +\ve}}\le \P(Z_1 >n) \le \frac{C(\ve)}{n^{\alpha -\ve}}
\]
Let $T_0 = \inf \{t >0, Z_1+\ldots Z_t = 0\}$. Then for all $\ve>0$
such that $\alpha -\ve >1$ and $\alpha +\ve<2$, there exist constants
$c'(\ve),C'(\ve)>0$ such that
\[
\frac{c'(\ve)}{n^{1/\alpha +\ve}}\le \P(T_0>n) \le \frac{C'(\ve)}{n^{1/\alpha 
-\ve}}.
\]
\end{lemma}
\begin{proof}
The proof follows from fairly elementary martingale arguments. We start with the 
upper bound. Let $S_n = \sum_{m \le n} Z_m$ and let $\tau_L = \inf \{ n \ge 1: 
S_n \ge L\}$ where $L$ will be chosen later. Set $T = T_0 \wedge \tau_L$. Then
$$
\P( T_0 \ge n) \le \P( \tau_L < T_0) + \P( S \text{ remains in $(0,L)$  for time 
$\ge n$}).
$$
We bound each term separately. Since $S$ stopped at $T$ is a nonnegative 
martingale we have (by Fatou's lemma), 
$$
L\P( \tau_L < T_0) \le \E(S_T) \le \E(S_0) = 1
$$
so $\P(\tau_L < T_0) \le 1/L$. On the other hand, for the second term we simply 
observe that at each step there is a probability at least $L^{-\alpha +o(1)}$ of 
leaving the interval $[0,L]$ hence
$$
\P( S \text{ remains in $(0,L)$  for time $\ge n$}) \le (1 - L^{- \alpha + 
o(1)})^n \le \exp ( - n L^{- \alpha + o(1)}).
$$
Therefore
$$
\P(T_0 \ge n) \le \frac1L  + \exp ( - L^{-\alpha + o(1)}n)
$$
Choosing $L = n^{1/\alpha - \eps}$, $\P(T_0 \ge n)  \le n^{-1/\alpha + \eps} + 
\exp ( - n^{\alpha\eps + o(1)} )= O( n^{-1/\alpha + \eps})$ as desired.

For the lower bound, we let $J = \inf\{ n \ge 1: Z_n \ge L\}$ where $L$ will be 
chosen later. Note that $\E( Z_J ) = O(L)$ hence if $T = T_0 \wedge \tau_L 
\wedge J$ then $S$ stopped at $T$ is uniformly integrable. Consequently, 
applying the optional stopping theorem at time $T$, we get
\begin{align*}
1& = \E( S_T)  = \E(S_T \indic{ J = \tau_L < T_0}) + \E(S_T \indic{\tau_L < J, 
\tau_L < T_0})\\
& \le \E( (L + Z_J) \indic{J = \tau_L < T_0}) + 2L \P( \tau_L < T_0)\\
& \le L \P( J < T_0) + \E(Z_J \indic{J< T_0}) + 2L \P(\tau_L< T_0).
\end{align*}
Note now that $J< T_0$ implies $\tau_L < T_0$ and that the event $J<  T_0$ is 
independent of $Z_J$, because it depends only on the values $Z_1, \ldots, 
Z_{J-1}$. Consequently, 
$$
1\le 3 L \P( \tau_L < T_0) + O(L) \P( J< T_0) \le O(L) \P( \tau_L < T_0).
$$
Therefore, $\P( \tau_L < T_0) \ge c/L$ for some constant $c>0$. We now choose $L 
= n^{1/\alpha + \eps}$. From \cref{lem:left_tail}, we see that 
if $\tau_L < T_0$ then 
the walk is super polynomially likely to remain positive for time at least $n$. 
 The lower bound follows. 
\end{proof}

We will also need a lemma which says that the sum of 
heavy-tailed random variables (with infinite expectation) is comparable to the 
maximum. 
\def \Var {\text{Var}}
\begin{lemma}\label{lem:sum_max}
 Let $\{X_i\}_{i\ge 1}$ be i.i.d. random variables with $\P(X_1 > k) 
= k^{-\alpha + o(1)}$ with $\alpha \in (0,1)$. Then for all $m \ge 1$, $\lambda 
> 0, \ve >0$
\[
 \P\left (\frac{\sum_{i=1}^{m} X_i }{(\max_{ 1 \le i  \le m}  X_i)^{1+\ve}} 
>\lambda \right) \le c(\ve)e^{-c'(\ve)\lambda}.
\]
In particular all moments of $\frac{\sum_{i=1}^{m} X_i }{(\max_{ 1 \le i  \le 
m} 
 X_i)^{1+\ve}}$ exist for all $\ve > 0$ and are uniformly bounded in $m\ge 1$.
\end{lemma}

\begin{proof}
For simplicity we will assume in this proof that the random variables $X_i$ are 
continuous so that the maximum is unique a.s. (This is no loss of generality, as 
we can always add a small continuous perturbation.) Fix $\ve>0$; we will allow 
every constant $c$ and implicit constants in $O$ notations below to depend on 
$\eps$ but nothing else. We will still write $c$ instead of $c(\eps)$ for 
simplicity.
 %

Let $ X_m^*:=\max_{ 1 \le i  \le m}  X_i $. Notice that if $k \le 
(m/\lambda)^{1/(\alpha+\ve)}$, $\P(X_m^* \le k) \le 
(1-ck^{-\alpha-\ve})^m \le e^{-c\lambda}$. Therefore we can restrict ourselves 
further to the event $\{X_m^* 
\ge (m/\lambda)^{\frac{1}{\alpha+\ve}}\}$. Conditionally on $ X^*_m = k$, where 
$k$ is some number larger than $(m/\lambda)^{1/(\alpha+\ve)}$, note that since 
$X_i$ are assumed to be continuous,
 $$
 \sum_{1 \le i \le m}X_i = k+\sum_{1 \le i \le m-1}\tilde X_i
 $$ 
 where $\tilde 
X_i$ are i.i.d.\ and has the law of $X_i$ conditioned to be at most $k$. It is 
a straightforward computation to see that $\E(\sum_{1 \le i \le m-1}\tilde 
X_i)=O(mk^{1-\alpha + \ve}) = O(k^{1+2\ve}\lambda)$ and $\Var(\sum_{1 \le i 
\le m}\tilde X_i) = O(mk^{2-\alpha + \ve}) = O(k^{2+2\ve}\lambda)$. 
Therefore using these bounds and Bernstein's inequality, for $k \ge 
(m/\lambda)^{\frac{1}{\alpha+\ve}}$:

\begin{align}
 \P\left (\sum_{i=1}^m X_i >\lambda k^{1+4\ve} | X_m^* = k\right) & \le 
\P\left(\sum_{i=1}^{m-1}\tilde X_i > (\lambda-1)k^{1+4\ve} \right) \nonumber\\
& \le \P(\sum_{i=1}^{m-1}(\tilde X_i - \E(\tilde X_i)) 
> \lambda k^{1+3\ve})\nonumber\\
& \le 2 \exp \left(-\frac{\lambda^2k^{2+6\ve}/2}{O(k^{2+2\ve}\lambda) 
 + \lambda k^{2+ 4\eps} /3  }  \right) \le 2 e^{-c \lambda }
\end{align}
which completes the proof.
\end{proof}

\small

\bibliographystyle{abbrv}
\bibliography{RS.bib}

\end{document}